\documentclass[sts,a4paper,preprint,reqno,11pt]{article}

\usepackage{setspace}
\usepackage{a4wide}
\usepackage[T1]{fontenc}
\usepackage[utf8]{inputenc}
\usepackage{dsfont}
\usepackage{mathrsfs}
\usepackage{amsmath}
\usepackage{amssymb}
\usepackage{verbatim}
\usepackage{amsthm}
\usepackage{graphicx}
\usepackage[usenames]{color}
\usepackage{fancybox}
\usepackage{enumitem}
\usepackage{eurosym}
\usepackage{cleveref}
\usepackage{autonum}
\usepackage[lined,longend,algoruled,linesnumbered,boxed]{algorithm2e}
\DontPrintSemicolon
\usepackage{algorithmic}

\usepackage[english]{babel}

\def\a{\alpha}
\def\b{\beta}
\def\g{\gamma}
\def\G{\Gamma}
\def\d{\delta}
\def\D{\Delta}
\def\e{\epsilon}

\def\t{\theta}
\def\T{\Theta}

\def\k{\kappa}
\def\l{\lambda}
\def\L{\Lambda}
\def\s{\sigma}

\def\ie{\textit{i.e., }}

\def\cf{\textit{cf. }}

\def\RR{\mathbb R}

\def\calA{\mathcal A}

\def\card{\textbf{Card}}
\def\fcar{\mathds{1}}

\def\suchthat{\,|\,}
	
\def\donc{\Rightarrow}

\def\esp{\mathbf E}
\def\var{\mathbf{Var}}

\def\prob{\mathbf P}

\theoremstyle{plain}
\newtheorem{theorem}{Theorem}
\newtheorem{lemma}{Lemma}
\newtheorem{proposition}{Proposition}
\newtheorem{corollary}{Corollary}
\newtheorem*{theorem*}{Theorem}
\newtheorem*{lemma*}{Lemma}
\newtheorem*{proposition*}{Proposition}
\newtheorem*{corollary*}{Corollary}
\theoremstyle{remark}

\newtheorem*{remark*}{Remark}

\newtheorem*{note*}{Note}
\theoremstyle{definition}

\newtheorem*{definition*}{Definition}

\def\bt{\boldsymbol \t}
\def\XX{\mathbb X}

\def\htau{\hat{\tau}}

\def\bY{{\mathbf Y}}
\def\bxi{{\boldsymbol \xi}}
\def\bu{{\boldsymbol u}}
\def\bw{{\boldsymbol w}}
\def\bX{{\boldsymbol X}}
\def\bZ{{\boldsymbol Z}}

\def\bnu{{\boldsymbol \nu}}
\def\beps{{\boldsymbol \epsilon}}

\newcommand{\vertiii}[1]{{\left\vert\kern-0.25ex\left\vert\kern-0.25ex\left\vert #1 
    \right\vert\kern-0.25ex\right\vert\kern-0.25ex\right\vert}}

\makeatletter
\newcommand{\customlabel}[2]{%
\protected@write \@auxout {}{\string \newlabel {#1}{{#2}{}}}}
\makeatother

\begin{document}

\title{\bf Estimation of the $\ell_2$-norm and testing in sparse linear regression with unknown variance}
\author{A. Carpentier, O. Collier, L. Comminges, A.B. Tsybakov and Y. Wang\thanks{A. Carpentier is with the University of Magdeburg, O. Collier is with Modal’X, Universit\'e Paris-Nanterre and CREST. L. Comminges is with CEREMADE, Universit\'e Paris-Dauphine and CREST. A.B. Tsybakov is with CREST, ENSAE, Institut Polytechnique de Paris. Y. Wang is with Tsinghua University and Shanghai Qi Zhi Institute.\newline
{\bf MSC 2010 subject classifications:} 62J05, 62G10.
{\bf Keywords and phrases:} linear regression, sparsity, signal detection.}}
\maketitle

\begin{abstract}
We consider the related problems of estimating the $\ell_2$-norm and the squared $\ell_2$-norm in sparse linear regression with unknown variance, as well as the problem of testing the hypothesis that the regression parameter is null under sparse alternatives
with $\ell_2$ separation. We establish the minimax optimal rates of estimation (respectively, testing) in these three problems.

\end{abstract}

\section{Introduction}

Assume that we have observations $(\mathbb{X},\bY)$ satisfying the model 
\begin{equation}\label{model}
	\bY= \mathbb{X}\bt + \s\bxi,
\end{equation}
where $\mathbb{X}$ is a $N\times p$ random matrix, $\bt\in \RR^p$ is an unknown parameter of interest, $\bxi \in \mathbb R^N$ is the random noise, and $\s>0$ is the unknown noise level. We assume that the entries ${X}_{ij}, i=1,\dots, N, j=1,\dots, p,$ and  $\xi_i, i=1,\dots, N,$  of matrix $\mathbb{X}$ and of vector $\bxi$ constitute a collection of $N(p+1)$ jointly independent zero mean random variables with variance $1$.

We use the notation $\prob_{\bt,\s}$ (or $\prob$ when there is no ambiguity) for the probability distribution of $(\mathbb{X},\bY)$ satisfying  model  \eqref{model}, and  $\esp_{\bt,\s}$ (or $\esp$ when there is no ambiguity) for the expectation with respect to this distribution. 

For a vector $\bu = (u_i)_{1\le i\le p} \in \mathbb R^p$, we define the "$\ell_0$-norm" $\|\cdot\|_0$, the $\ell_1$-norm $\|\cdot\|_1$ and the $\ell_2$-norm  $\|\cdot\|_2$ by the relations
$$
\|\bu\|_0 = \sum_{i=1}^p \fcar\{u_i \ne 0\}, \quad  \|\bu\|_1 = \sum_{i=1}^p |u_i|, \quad \|\bu\|^2_2 = \sum_{i=1}^p u_i^2,
$$ 
where $\fcar{\{\cdot\}}$ is the indicator function.  

We assume that $\bt$ belongs to a set of $s$-sparse vectors defined as follows:  
$$
	B_0(s)  = \{ \bu \in \mathbb R^p : \|\bu\|_0 \leq s\},
$$
where $s$ is an integer such that $1 \le s \le p$.  

In this paper, we consider the related problems of estimating the $\ell_2$-norm and the squared $\ell_2$-norm of $\bt$, as well as testing the null hypothesis that $\bt=0$  under sparse alternatives
with $\ell_2$ separation (signal detection in $\ell_2$-norm). 

\medskip

{\bf Estimation of $\|\bt\|_2$ and $\|\bt\|_2^2$.}  
We consider the problem of minimax optimal estimation  of the norm $\|\bt\|_2
$ and of  the squared norm $
Q(\bt) = \|\bt\|_2^2$ on the class of $s$-sparse vectors $B_0(s)$ when the noise level $\s$ is unknown. Namely, for a fixed $\d \in (0,1)$, we find estimators $\hat \L$ of $\|\bt\|_2$ such that 
\begin{equation}\label{eq:est}
	\sup_{\s>0} \ \sup_{\bt \in B_0(s)} \prob_{\bt, \s}(|\hat \L- \|\bt\|_2 | > \phi_*) \leq \d,
\end{equation}
where 
$$\phi_*=\phi_*(p,N,s,\s, \delta)>0,$$
is the non-asymptotic minimax optimal rate. That is, there exists a constant $c_\delta>0$ that depends only on $\delta$  such that 
\begin{equation}\label{eq:est:lower}
	\inf_{\hat T}\sup_{\s>0} \ \sup_{\bt \in B_0(s)} \prob_{\bt, \s}(|\hat T- \|\bt\|_2| > c_\d \phi_*) \geq \d.
\end{equation}
Here and below, $\inf_{\hat T}$ denotes the infimum over all estimators. Next,  for a fixed $\delta \in (0,1)$, we find estimators $\hat Q$ of $Q(\bt)$ such that
\begin{equation}\label{eq:est2}
	\sup_{\s>0} \ \sup_{\bt \in B_0(s)} \prob_{\bt,\s}(|\hat Q- Q(\bt)| > q_{\bt}) \leq \delta,
\end{equation}
where the scaling 
$$q_{\bt}=q_{\bt}(p,N,\s, \delta)>0,$$ 
depends on $\bt$ only through the $\ell_2$ and $\ell_0$-norms $\|\bt\|_2$ and $\|\bt\|_0$. We find the minimax optimal scaling 
$q_{\bt}^*=q_{\bt}^*(p,N,\s, \delta)>0$ in the sense that there exists $c_\delta>0$ that depends only on $\delta$ such that 
\begin{equation}\label{eq:est:lower2}
	\inf_{\hat T}\sup_{\s>0} \ \sup_{\bt \in B_0(s)} \prob_{\bt, \s}(|\hat T- Q(\bt)| > c_\d q_{\bt}^*) \geq \d,
\end{equation}
and
\begin{equation}\label{eq:est21}
	\inf_{\hat T}\sup_{\s>0} \ \sup_{\bt \in B_0(s)} \prob_{\bt,\s}(|\hat T- Q(\bt)| > q_{\bt}^*) \leq \delta.
\end{equation}
Furthermore, as $q_{\bt}^*$ is an increasing function of the $\ell_2$-norm $\|\bt\|_2$, it is natural to consider the minimax setting on the intersection $B_0(s)\cap \{\|\bt\|_2\le \kappa\}$ with $\kappa>0$, for which we prove that the rate 
$$q^*_{s,\kappa} =q^*_{s,\kappa}(p,N,\s, \delta):=\max \{q_{\bt}^*(p,N,\s, \delta): \ \|\bt\|_0 \leq s,\|\bt\|_2\le \kappa\},$$ 
is minimax optimal. All the rates in this paper are obtained either in non-asymptotic form or for $N\ge N_0$, where $N_0$ does not depend on $p,s,\s$.


\medskip

{\bf Signal detection under $\ell_2$ separation.} 
When $\s$ is unknown, we consider the problem of testing the hypothesis $H_0:\bt=0$ against the sparse alternative $H_1: \bt \in B_0(s), \|\bt\|_2\geq  \s\rho$, where $\rho>0$.
 For a test $\Delta=\Delta(\mathbb{X},\bY)$, i.e., a measurable function of $(\mathbb{X},\bY)$ with values in $\{0,1\}$,  we introduce the risk $R(\Delta,p,N,s,\rho)$ as the sum of the type I and type II error probabilities:
\begin{equation}\label{eq:test}
	R(\Delta,p,N,s,\rho)= \sup_{\s>0}\prob_{\mathbf 0,\s}(\Delta=1) + \sup_{\sigma>0} \ \sup_{\bt \in B_0(s), \|\bt\|_2\geq  \s\rho} \prob_{\bt,\s}(\Delta=0).
\end{equation}
For a fixed $\d\in (0,1)$, we define the minimax separation rate $\rho_*=\rho_*(p,N,s, \delta)>0$ as
$$
\rho_*:= \rho_*(p,N,s, \delta):= \inf\{\rho>0: \ \inf_{\bar \Delta} R(\bar\Delta,p,N,s,\rho) \le \d \}.
$$
Here and below, $\inf_{ \bar\Delta}$ denotes the infimum over all tests $\bar\Delta$. We find a minimax rate optimal test $\Delta$, i.e., a test such that $R(\Delta,p,N,s,\rho) \le \d$ for all  $\rho\ge c_\delta\rho_*$ where $c_\delta>0$ is a constant that depends only on $\delta$. We obtain the non-asymptotic expression for the separation rate
$\rho_*$. 

\section{Related work}

The problems of estimating the $\ell_2$-norm and of signal detection under $\ell_2$ separation are closely related to each other. We provide here an overview of the previous results for both problems.

\paragraph{Signal detection.} 

 In the Gaussian mean model, which corresponds to an orthogonal non-random design $\mathbb X$, the problem of signal detection has been extensively studied in the last fifteen years (see, e.g.,~\cite{ingster1997, ingster_suslina, baraud02,jin2004, Collier2017,collier2016optimal} and the references therein). More recently, this problem has also been investigated in the random design linear regression model, which is most related to the present paper~\cite{2010_EJS_Ingster, 2011_AS_Arias-Castro, verzelen_minimax, carpentier2018minimax}. Among these, \cite{2011_AS_Arias-Castro} and \cite{carpentier2018minimax} deal with the case of known $\s$, while~\cite{2010_EJS_Ingster, verzelen_minimax} consider both known and unknown $\s$.
The setting in \cite{2011_AS_Arias-Castro} is different from ours since it does not consider the alternative defined by separation in the $\ell_2$-norm.
It does not allow one to compare \cite{2011_AS_Arias-Castro} directly with~\cite{2010_EJS_Ingster, verzelen_minimax, carpentier2018minimax} and with the present work. Note also that  \cite{2011_AS_Arias-Castro} explores an asymptotic setting as $p,N,s$ tend to $\infty$, 
while in this paper we are interested in non-asymptotic results. The setting is also asymptotic in~\cite{2010_EJS_Ingster} where it is 
assumed that $p,N, s$ tend to $\infty$ in such a way that 
$s=p^a$ for some $0<a<1$, and $s\log(p)=o(N)$.  It can be deduced from the asymptotic argument in~\cite{2010_EJS_Ingster}  that in this regime, and if the matrix $\mathbb X$ has i.i.d. standard normal entries and $\bxi$ is standard normal, the minimax separation rate satisfies, in the case of {\it known} $\s$, the relations:
\begin{align}
\rho_* &\asymp  \min \Big(\sqrt{\frac{s\log(p)}{N}},  N^{-1/4} \Big) \quad \text{if} \ 0<a<1/2,\label{ejs}\\
\rho_* &\asymp  \min \Big(\frac{p^{1/4}}{\sqrt{N}},  N^{-1/4}   \Big) \qquad \quad \ \ \text{if} \ 1/2<a<1.
\end{align}
A refined analysis of the minimax risk asymptotics for testing in this model with known $\s$ is given in \cite{mukherjee2020}.

Next, in the setting with {\it unknown} $\s$,~\cite{2010_EJS_Ingster} obtained that, in the same asymptotic regime, the asymptotic minimax separation rate is of the order of $\sqrt{\frac{s\log(p)}{N}}$ when $0<a<1/2$. 

More insight into the behavior of tests under {\it unknown} $\s$ is provided in~\cite{verzelen_minimax}, which is the first paper where non-asymptotic separation rates were explored in the linear regression context. Namely, considering the case of Gaussian $\mathbb X$ and the  configuration of parameters where $p>N$ and $s<p^a$ with $0<a<1/2$,~\cite{verzelen_minimax} proves that the minimax separation rate $\rho_*$ is of the order of $\sqrt{\frac{s\log(p)}{N}}$ whenever $p\lesssim N^{1/a}$, where the sign $\lesssim$ means inequality up to a logarithmic factor. On the other hand,~\cite{verzelen_minimax} shows that if $p\gtrsim N^{1/a}$ and $s=p^a$ with $0<a<1/2$, the rate $\rho_*$  explodes becoming at least greater than a constant independent of $p,N,s$. This is proved for the case of $\mathbb X$ with i.i.d. standard normal entries and $\bxi$ standard normal. A conclusion is that there is no much interest in considering signal detection in the regime $p>N^b$ with $b>2$ when $\s$ is unknown. Regarding the case of {\it known} $\s$,~\cite{verzelen_minimax} proves that, for $\mathbb X$ with i.i.d. standard normal entries and $\bxi$ is standard normal, the non-asymptotic rate is
\begin{align}
\rho_* &\asymp  \min \left(\sqrt{\frac{s\log(p)}{N}},  N^{-1/4} \right), \quad \text{if} \ p\ge N^2, s\le N,
\end{align}
which is a non-asymptotic version of the corresponding result in~\cite{2010_EJS_Ingster} providing a refinement of \eqref{ejs} in the zone $p\ge N^2$ in terms of conditions on $p,N,s$.

Finally,~\cite{carpentier2018minimax} considers the setting with {\it known} $\s$ and establishes the following non-asymptotic lower bound valid for all configurations of $p,N,s$:
\begin{align}\label{carpentier}
\rho_* &\ge  c\min \Big(\sqrt{\frac{s \log\big(1+\sqrt{p}/s\big)}{N}},  \frac{p^{1/4}}{\sqrt{N}}, N^{-1/4} \Big)
\end{align}
where $c>0$ is a constant. Furthermore,~\cite{carpentier2018minimax} proves that this bound is attained in the the low-dimensional regime $p\le \gamma N$ if $\mathbb X$ is a matrix with i.i.d. standard normal entries and $\bxi$ is standard normal. Here $\gamma\in (0,1)$ is an absolute constant specified in~\cite{carpentier2018minimax}. Noticing that for $p\le N$ the expression in \eqref{carpentier} reduces to its first term, the results of~\cite{carpentier2018minimax} lead to the conclusion that, for {\it known} $\s$, the non-asymptotic minimax separation rate satisfies 
\begin{align}\label{low-dim-rate}
\rho_* &\asymp  \sqrt{\frac{s \log\big(1+\sqrt{p}/s\big)}{N}} \qquad  \text{if} \ p\le \gamma N
\end{align}
when $\mathbb X$ is a matrix with i.i.d. standard normal entries and $\bxi$ is standard normal. In what follows, the regime $p\le \gamma N$ where $0<\gamma<1$ will be called the {\it low-dimensional  regime}.

\paragraph{Estimation of $\|\bt\|_2$ and $\|\bt\|_2^2$.} 

The problem of estimating $\|\bt\|_2$ is closely related to signal detection under $\ell_2$ separation. Indeed, any estimator of $\|\bt\|_2$ can be used as a test statistic and thus it is easy to check that the minimax separation rate for tests is not of greater order than the minimax rate of estimating $\|\bt\|_2$. In other words, the problem of estimating $\|\bt\|_2$  is more difficult than of signal detection. Intuitively, it is clear since in signal detection the objective is only to decide whether $\|\bt\|_2$ is large or zero.

Estimation of $\|\bt\|_2$ and $\|\bt\|_2^2$ is well understood in various settings of the Gaussian mean model - see, for example, \cite{donoho_nussbaum, MR2253108, comminges2013minimax,  Collier2017, CCNT2018} and the references therein. In particular, non-asymptotic minimax rates of estimating $\|\bt\|_2$ and $\|\bt\|_2^2$ in the Gaussian mean model on the sparsity classes $B_0(s)$ are established in \cite{Collier2017}
for the case of known $\s$, while \cite{CCNT2018} derives such rates for estimating $\|\bt\|_2$ in the case of unknown $\s$. It turns out that the rates are quite different in the two cases. Moreover, 
\cite{CCNT2018} obtains the minimax rates of estimating $\|\bt\|_2$ when the distribution of $\bxi$ belongs to a class of laws with polynomially or exponentially decreasing tails, which are also quite different from those for the Gaussian noise.

In the linear regression model, much less is known and only partial results are available. One may immediately notice the  following. As there exist estimators $\hat \bt$ achieving  the rate $\s\sqrt{s\log(ep/s)/N}$ in the $\ell_2$-norm whenever $s\log(ep/s)/N\le \k_0$ where $\k_0>0$ is a small enough constant (see, e.g., \cite{2012_Sun, BellecLecueTsybakov2017, Derumigny2017}), a direct plug-in estimator $\|\hat \bt\|_2$ of $\|\bt\|_2$ achieves the same rate. This gives the following upper bound on the minimax rate $\phi_*$ of estimation of $\|\bt\|_2$:
$$
\phi_* \le c\s\sqrt{s\log(ep/s)/N}
$$
for a constant $c>0$. This bound is valid with known or unknown $\s$ provided $\mathbb{X}$ has independent subGaussian rows. However, this bound is not tight. Indeed, it is shown in \cite{carpentier2018minimax} that, for the case of {\it known} $\s$, the non-asymptotic minimax rate satisfies
\begin{align}\label{low-dim-rate-functional}
\phi_* \asymp \s\sqrt{\frac{s \log\big(1+\sqrt{p}/s\big)}{N}} \qquad  \text{if} \ p\le \gamma N \ \text{(low-dimensional regime)},
\end{align}
when $\mathbb X$ is a matrix with i.i.d. standard normal entries and $\bxi$ is standard normal.
Note that the result \eqref{low-dim-rate-functional} slightly improves upon the plug-in rate $\sqrt{s\log(ep/s)/N}$ in the {\it sparse zone} $s\le \sqrt{p}$ and substantially improves upon this rate in the {\it dense zone} $s> \sqrt{p}$. Indeed, it is not hard to check that
\begin{align}\label{equiv}
\sqrt{\frac{s \log\big(1+\sqrt{p}/s\big)}{N}} \asymp \frac{p^{1/4}}{\sqrt{N}} \qquad  \text{if} \ s> \sqrt{p}.
\end{align}
Next, under the same conditions (low-dimensional regime $p\le \gamma N$ with known $\s$),  \cite{carpentier2018minimax} establishes the following non-asymptotic upper bound on the rate of estimation of $\|\bt\|_2^2$:
\begin{equation}\label{carpentier}
		q_{\bt} \leq c\left(  \frac{ \s^2 s \log\big(1+\sqrt{p}/s\big)}{N} + \frac{\s \|\bt\|_2}{N}\right)
	\end{equation}
	for a constant $c>0$ depending only $\d$. 
However,  \cite{carpentier2018minimax}  does not explore whether this bound is optimal. 

The recent paper \cite{CaiGuo2016genetic_relatedness} considers estimation of $\|\bt\|_2^2$  in the {\it highly sparse zone}: $s<p^a$ for some $0<a<1/2$ when $\s$ is unknown and when the additional condition $\sqrt{s\log(p)/N}=O(\|\bt\|_2)$ holds.  Under the assumptions  that matrix $\mathbb{X}$ has independent Gaussian rows with zero mean and covariance matrix $\Sigma$ and that $s\log(p)/N\le c_0$ where $c_0>0$ is a small enough constant, \cite{CaiGuo2016genetic_relatedness} proves the following non-asymptotic upper bound 
\begin{equation}\label{guo}
		q_{\bt} \leq c(\s)\left( \max( \|\bt\|_2,1) \frac{  s \log(p)}{N} + \frac{\|\bt\|_2}{N}\right) \qquad  \text{if} \ s< p^a, \ 0<a<1/2,
	\end{equation}
	for a constant $c(\s)>0$ depending only $\d$ and $\s$. Note that, disregarding the dependence on $\s$ and on the constants, the bound  \eqref{carpentier} is better than \eqref{guo}. 
	
In the present paper, we derive non-asymptotic minimax rates for both  functionals, $\|\bt\|_2$ and $\|\bt\|_2^2$. This is done both in the low-dimensional and high-dimensional ($p> \gamma N$) regimes and both in the dense and sparse zones. We focus on the case of unknown $\s$. 

One of our conclusions is that the bound \eqref{guo} is not rate optimal {\it when $\Sigma$ is the identity matrix} and in this case the minimax rate is driven by \eqref{carpentier}. On the other hand, \cite{CaiGuo2016genetic_relatedness} establishes that \eqref{guo} is rate optimal when the minimax setting is considered over a sufficiently large class of matrices $\Sigma$, and the additional condition $\sqrt{s\log(p)/N}=O(\|\bt\|_2)$ holds. In other words, the worsening of the rate observed in \eqref{guo} should be explained by the complexity of matrix $\Sigma$ and not by the complexity of $\t$ or the fact that $\s$ is unknown. 


A related work~\cite{verzelen-gassiat} deals with estimation of yet another quantity $\|\Sigma^{1/2}\bt\|_2^2/\sigma^2$ (called the the signal-to-noise ratio) and establishes the minimax rates for this problem on the sparsity class $B_0(s)$ that are quite different from the minimax rates for estimation of $\|\bt\|_2^2$. In particular, unlike \eqref{carpentier} or \eqref{guo} those rates show no dependence on $\|\bt\|_2^2$.

\section{Summary of the results}

The main results of the present paper can be summarized, in a simplified form, as follows.
\begin{itemize}
	\item {\bf Signal detection.} We propose a test $\Delta$ such that under suitable conditions on $\mathbb X$ and~$\bxi$ the following holds.  Let $\delta\in (0,1)$. If $s\log(ep/s)/N \leq \k_0$ where $\k_0>0$ is a small enough constant depending on $\delta$,  then  $R(\Delta,s,\rho)\leq \delta$ whenever
	\begin{equation}
		\rho \geq c_\delta\sqrt{ \frac{s \log\big(1+\sqrt{p}/s\big)}{N}}
	\end{equation}
	where  $c_\delta>0$ is a constant depending only on $\delta$. Moreover, if $\mathbb X$ is a matrix with i.i.d. standard normal entries and $\bxi$ is standard normal, we show that, under the same conditions, the non-asymptotic minimax separation rate in the setting with {\it unknown} $\s$ has the form 
	\begin{equation}\label{rate-test-unknown-sigma}
		\rho_* \asymp \sqrt{ \frac{s \log\big(1+\sqrt{p}/s\big)}{N}}. 
	\end{equation}
	Remark that for $s= p$, the condition $s\log(p/s)/N \leq \k_0$ reduces to $p\leq \k_0 N$, so that only the low-dimensional case is covered. Moreover, since the function $s\mapsto s\log(ep/s)/N$ is increasing for $1\le s\le p$ the condition $s\log(p/s)/N \leq \k_0$ in the dense zone $s> \sqrt{p}$ implies that  $p\leq cN^2/(\log p)^2$ with some constant $c>0$ that does not depend on $p,s,N$. Let us emphasize that this restriction
	 $p\leq cN^2/(\log p)^2$ is rather natural since for $p\ge N^2$ we get into the divergence regime described in \cite{verzelen_minimax}. Indeed,  due to \eqref{equiv}, for $p\ge N^2$ the value $\rho_*$ in \eqref{rate-test-unknown-sigma} becomes greater than a constant independent of $p,s,N$ in the dense zone $s> \sqrt{p}$. As follows from Theorem \ref{th:lower1} below, the minimax risk under unknown $\s$ is also greater than a constant for $p\ge N^2$ and $s> \sqrt{p}$. Therefore, considering $p\ge N^2$ is not of much statistical interest since in this regime increasing the sample size does not improve the quality of testing. The same remark applies to the problems of estimation of $\|\bt\|_2$ and of $\|\bt\|_2^2$.
	
	\item {\bf Estimation of $\|\bt\|_2$. }  We propose an estimator $\hat \L$ such that, under the same conditions as in the previous item,
	the bound~\eqref{eq:est} is satisfied with
	\begin{equation}
		\phi_* \geq c_\delta\s \sqrt{ \frac{s \log\big(1+\sqrt{p}/s\big)}{N}}
	\end{equation}
	where  $c_\delta>0$ is a constant depending only on $\delta$. Moreover, if $\mathbb X$ is a matrix with i.i.d. standard normal entries and $\bxi$ is standard normal, we show that, under the same conditions, the non-asymptotic minimax estimation rate on the class $B_0(s)$ in the setting with  {\it unknown} $\s$ has the form 
	\begin{equation}\label{rate-est-norm}
		\phi_* \asymp \s \sqrt{ \frac{s \log\big(1+\sqrt{p}/s\big)}{N}}.
	\end{equation}
%
	\item {\bf Estimation of $\|\bt\|_2^2$.} We propose an estimator $\hat Q$ such that, under the same conditions as in the previous two items,
	the bound~\eqref{eq:est2} is satisfied with $q_{\bt}$ as in \eqref{carpentier}
	where $c>0$ is a constant depending only on $\delta$. Furthermore, if $\mathbb X$ is a matrix with i.i.d. standard normal entries and $\bxi$ is standard normal, we show that, under the same conditions, the non-asymptotic minimax estimation rate on the class $B_0(s)\cap \{\bt:  \|\bt\|_2\le \kappa\}$ in the setting with {\it unknown} $\s$ has the form
	\begin{equation}\label{rate-est-squared-norm}
		q^*_{s,\kappa}\asymp \min\Big(   \frac{\s^2  s \log\big(1+\sqrt{p}/s\big)}{N} + \frac{\s \kappa}{\sqrt{N}}, \ \kappa^2\Big).
	\end{equation}
	%
	\end{itemize}

\section{Assumptions and notation}

If $\mathbf{u},\mathbf{v}$ are two vectors, we denote by $\langle \mathbf u, \mathbf v \rangle=\mathbf u^T \mathbf v$ their inner product. If $\mathbb{A}$ is a matrix, $\mathbb{A}^T$ denotes its transpose, $\|\mathbb{A}\|_{\infty}$ its spectral norm and  $\|\mathbb{A}\|_F$ its Frobenius norm, $\mathbf{A}_j$ denotes the $j$th column of $\mathbb{A}$, and $\mathbf{A}^j$ its $j$th row. Throughout the paper, we use blackboard bold symbols (e.g., $\mathbb{A}$) to denote matrices, bold symbols (e.g., $\mathbf{A}$) to denote vectors, and light symbols (e.g., $A$) to denote scalars. In particular, $A_j$'s are the entries of vector $\mathbf{A}$ and $A_{ij}$'s are the entries of matrix $\mathbb{A}$.  Given a symmetric matrix $\mathbb{A}$, $\lambda_{\min}(\mathbb{A})$ and $\lambda_{\max}(\mathbb{A})$ stand for its largest and smallest eigenvalue, respectively. We denote by $\mathbb{I}_p$ the identity matrix in dimension $p$. 
We denote by $C,C'$ positive constants whose values can differ on different occurrences. 

We will assume depending on the considered method that we have either $N=2n$, or $N=3n$ 
for an integer $n$, and the sample $(\XX,\bY)$ is divided into either two or three 
sub-samples of size $n$ that we denote $(\XX_1,\bY_1)$, $(\XX_2,\bY_2)$, and $(\XX_3,\bY_3)$.  
In each part of the paper, we specify how large $N$ should be (\ie $2n$ or $3n$). Since we are interested only in the rates of convergence, taking  $N$ in this form is without loss of generality apart from imposing the conditions $N\ge 2$ or $N\ge 3$, 
respectively. 

\paragraph{Assumptions.}  
  
We assume throughout the paper that $p\ge 2$ and the entries ${X}_{ij}, i=1,\dots, N, j=1,\dots, p,$ and  $\xi_i, i=1,\dots, N,$  of matrix $\mathbb{X}$ and of vector $\bxi$ constitute a collection of $N(p+1)$ jointly independent zero mean random variables with variance $1$.
In addition, in different parts of the paper we will refer to some conditions from the following list. \\

\noindent
{\bf Condition on $p,s,N$.}
\begin{enumerate}
	\item[$(P)$]\customlabel{6}{$(P)$}   It holds that $s \log(ep/s)/N\le \k_0$ for some small enough positive constant $\k_0$. 
\end{enumerate}

\noindent
{\bf Conditions on $\bxi$.}
\begin{enumerate}
	\item[$(\bxi^*1)$]\customlabel{4}{$(\bxi^*1)$}  $\esp (\xi_i^8)\le c_8$ with $c_8>0$ for $i=1, \dots,N$. 
	\item[$(\bxi^*2)$]\customlabel{1}{$(\bxi^*2)$} Random variables $\xi_i$, $i=1, \dots,N$, are $L$-subGaussian for some constant $L> 0$, that is
	$\esp (\exp(t \xi_i) )\le \exp(L^2t^2/2)
	$
	for all $t\in \RR$. 
\end{enumerate}

\noindent
{\bf Conditions on $\mathbb X$.}
\begin{enumerate}
	\item[$(\mathbb X^*1)$]\customlabel{2}{$(\mathbb X^*1)$} Random variables $X_{ij}$ have a density with respect to the Lebesgue measure bounded by a constant $K> 0$ for $i=1,\dots, N, j=1,\dots, p$.  
	\item[$(\mathbb X^*2)$]\customlabel{3}{$(\mathbb X^*2)$} $\esp (X_{ij}^4)\le d_4$ with $d_4>0$ for  $i=1,\dots, N, j=1,\dots, p$.
	\item[$(\mathbb X^*3)$]\customlabel{5}{$(\mathbb X^*3)$} Random variables $X_{ij}$, $i=1,\dots, N, j=1,\dots, p$, are $M$-subGaussian for some constant $M> 0$.
\end{enumerate}
  
\section{Generic estimators}\label{sec:general}

In this section, we  present generic estimators that will be used throughout the paper.
First, we consider some preliminary estimator 
$\hat{\bt}=(\hat \t_1,\dots,\hat \t_p)$  
based on the subsample $(\XX_1,\bY_1)$. 
Then, using the second subsample $(\XX_2,\bY_2)$  we define the following estimators of the values~$\t_j^2$:
\begin{align}
	\begin{split}\label{aj}
		a_j(\hat{\bt})&:=\hat{\t}_j^2+\frac{2\hat{\t}_j}{n}\mathbf X_{2,j}^T(\bY_2-\XX_2\hat{\bt})\\
		&+\frac{1}{n(n-1)}\sum_{k\neq l} X_{2,kj} X_{2,lj} \big(Y_{2,k}-\mathbf X_{2}^k\hat{\bt}\big)\big(Y_{2,l}-\mathbf X_{2}^l\hat{\bt}), \qquad 1\le j\le p.
	\end{split}
\end{align}
Here, $X_{2,j}$ and $X_{2,kj}$ are the $j$th column and the $(k,j)$th entry of matrix $\XX_2$, respectively, and $Y_{2,k}$ is the $k$th entry of vector $\bY_2$.
The estimator $a_j(\hat{\bt})$ can be viewed as a centered version of $\tilde{\theta}_j^2$ where $\tilde{\theta}_j$ is the  $j$th component of the conditionally unbiased estimator
\begin{align}\label{def:bttilde}
	\tilde{\bt} &:= \hat{\bt}+\frac{1}{n}\XX_2^T(\bY_2-\XX_2\hat{\bt}) \\
				&= \bt + \Big[\frac{1}{n}\XX_2^T\XX_2-\mathbb{I}_p\Big](\bt-\hat{\bt}) + \frac{\s}{n}\XX_2^T \bxi_2.
\end{align}
Note that $\tilde{\bt}$ is an unbiased estimator of $\bt$ if we consider $\hat{\bt}$ as fixed.

\subsection{Generic estimator for the dense case ($s> \sqrt{p}$)}

Given an estimator $\hat{\bt}$ based on the first subsample $(\XX_1,\bY_1)$, we define an estimator of $Q(\bt)$ as
\begin{align}
	\hat{Q}_{D}(\hat{\bt}):=& \sum_{j=1}^p a_j(\hat{\bt}).\label{Qhatgeneral}
\end{align}
The following general result allows one to link the risk of  $\hat{Q}_{D}(\hat{\bt})$ to the risk of the preliminary estimator $\hat{\bt}$.
\begin{theorem}\label{theo:general}
	Assume that Condition~\ref{3} holds. There exists a constant $C>0$ depending only on $d_4$ such that, for all $\bt\in\RR^p$, 
	\begin{equation}
		\esp_{\bt,\s}\Big[\big(\hat{Q}_{D}(\hat{\bt})-Q(\bt)\big)^2\mid \hat \bt\Big]\le C \Big[\frac{\|\bt\|_2^2}{n}\big(\|\hat{\bt}-\bt\|_2^2+\s^2\big)+\frac{p}{n^2}\big(\|\hat{\bt}-\bt\|_2^4+\s^4\big)\Big].
	\end{equation}		
\end{theorem}
In what follows, we will use $\hat{Q}_D(\hat{\bt})$ with two different preliminary estimators $\hat{\bt}$ depending on whether we are in the low-dimensional or high-dimensional regime. 

\subsection{Generic estimator for the sparse case ($s\leq \sqrt{p}$)}

In the sparse case, along with the first preliminary estimator $\hat{\bt}$ we consider a second one that we denote by $\bar{\bt}$.  Given a matrix $\mathbb M$ of dimension $p\times p$, two preliminary estimators $\hat{\bt}$ and $\bar{\bt}$, an estimator $\hat{\s}$ of $\s$, and a constant $\a$, we define an estimator of $Q(\bt)$ as
\begin{equation}\label{def:QS}
	\hat{Q}_{S}(\hat{\bt},\bar{\bt},\hat{\s},\mathbb M,\a) := \sum_{j=1}^p a_j(\hat{\bt})\fcar\Big\{|\bar{\theta}_j|>\a\hat{\s} \sqrt{M_{jj} \log(1+p/s^2)}\Big\}.
\end{equation}
The risk of this estimator depends on the distribution of~$\bar{\bt}$. Thus, it is difficult to obtain a result similar to Theorem~\ref{theo:general}, that is to bound the risk of $\hat{Q}_{S}$ by a function of~$\|\hat{\bt}-\bt\|_2$ for any estimator $\hat{\bt}$. 
Nevertheless, the risk of $\hat{Q}_S$ for specific choices of estimators can be readily controlled. It turns out that if we find estimators $\hat{\bt}$ and $\hat \s$ such that $\|\hat{\bt}-\bt\|_2$ and $|\hat{\s}-\s|$ are of the order $\s$ with high probability, and an estimator $\bar \bt$ satisfying $\bar \bt \sim \mathcal N(\bt, \s^2 \mathbb I_p/n)$, then  with high probability
\begin{equation}
	\big(\hat{Q}_{S}(\hat{\bt},\bar{\bt},\hat{\s},\mathbb I_p,\a)-Q(\bt)\big)^2\le C \Big[\s^2\frac{\|\bt\|_2^2}{n}+\s^4\frac{s^2\log^2(1+\sqrt{p}/s)}{n^2}\Big],
\end{equation}
which is the desired bound.
In the following, we construct estimators $\hat{\bt}, \hat \s$ with such a property and an estimator $\bar{\bt}$ distributed approximately as $\mathcal N(\bt, \s^2 \mathbb I_p/n)$.

\section{Upper bounds for the low-dimensional regime: $p \leq \g n$ with $0< \g < 1$}\label{sec:plen}

In this section, we assume that $N=2n$ and we divide the sample in two subsamples of size~$n$, namely $(\XX_1,\bY_1)$ and $(\XX_2,\bY_2)$. We also assume that there exists a constant $\g\in (0,1)$ such that $p\le \g n $ (the low-dimensional regime).

In the low-dimensional regime, we take as a preliminary estimator $\hat \bt=\hat{\bt}_{OLS}$ where $\hat{\bt}_{OLS}$ is the least squares estimator based on the first subsample  $(\XX_1,\bY_1)$. If Condition~\ref{2} holds, then $\XX_1^T\XX_1$ is almost surely invertible, so that the least squares estimator can be almost surely written in the form 
\begin{equation}\label{eq:OLS}
	\hat{\bt}_{OLS}= \big(\XX_1^T\XX_1\big)^{-1}\XX_1^T\bY_1.
\end{equation}
We also consider the following standard estimator of $\s$ based on $\hat{\bt}_{OLS}$: 
\begin{equation}\label{eq:estsigOLS}
	\hat{\s}_{ OLS} = \frac{\|\bY_1-\XX_1\hat{\bt}_{OLS}\|_2}{\sqrt{n-p}}.
\end{equation}

\subsection{Estimation of $Q(\bt)$ and $\|\bt\|_2$ in the dense case: $s > \sqrt{p}$}

In the dense case $s> \sqrt{p}$ of the low-dimensional regime, we use the estimators
\begin{equation}
	\hat Q_D^{LD} := \hat Q_D(\hat{\bt}_{OLS}), \quad \hat{\L}_D^{LD} := \Big| \hat Q_D^{LD}  \Big|^{1/2},
\end{equation}
where $Q_D$ and $\hat{\bt}_{OLS}$ are defined in~\eqref{Qhatgeneral} and~\eqref{eq:OLS} respectively. Using Theorem~\ref{theo:general} we obtain the following result. 
\begin{theorem}\label{theo:plen_dense}
	Let $1\le s\le p$ and let Conditions~\ref{4},~\ref{2}, \ref{3} hold. Then there exist constants $C>0$, $\g\in (0,1)$ depending only on $K,c_8,d_4$ such that, for $p\le \min(\g n, n-14)$,
	\begin{equation}
		\forall\bt\in\RR^p, \forall \s>0, \quad \esp_{\bt,\s}\big[\big(\hat Q_D^{LD}-Q(\bt)\big)^2\big]\le C\Big(\s^4\frac{p}{n^2}+\s^2\frac{\|\bt\|_2^2}{n}\Big),
	\end{equation}			
	and
	\begin{equation}
		\forall\bt\in\RR^p, \forall \s>0, \quad \esp_{\bt,\s} \big[\big(\hat{\L}_D^{LD}-\|\bt\|_2\big)^2\big]\le C\s^2\frac{\sqrt{p}}{n}.
	\end{equation}
\end{theorem}
It follows from the Markov inequality that, under  the assumptions of Theorem~\ref{theo:plen_dense}, for any $\d\in (0,1)$ the estimators $\hat Q=\hat Q_D^{LD}$ and $\hat \L=\hat{\L}_D^{LD}$ satisfy 
\begin{align}
	&\forall\bt\in\RR^p, \forall \s>0, \quad \mathbf P_{\bt,\s} \Big[\big(\hat{\L}_D^{LD}-\|\bt\|_2\big)^2 \geq \frac{C\s^2}{\d}\frac{\sqrt{p}}{n} \Big] \leq \delta
	\\
	&\forall\bt\in\RR^p, \forall \s>0, \quad \mathbf P_{\bt,\s} \Big[\big(\hat Q_D^{LD}-Q(\bt)\big)^2 \geq \frac{C}{\d}\Big(\s^4\frac{p}{n^2} + \s^2\frac{\|\bt\|_2^2}{n}\Big)\Big] \leq \delta.
\end{align}
Note also that the zero estimator $\hat Q \equiv 0$ satisfies  \eqref{eq:est2} with $q_{\bt}(p,N,\s,\d)\le \|\bt\|_2^2$. Thus, we have the following corollary.
 \begin{corollary}\label{cor:LDdense}
Let the assumptions of Theorem~\ref{theo:plen_dense} hold. Then, for any $\d\in (0,1)$ the estimators  $\hat \L=\hat{\L}_D^{LD}$ and $\hat Q=\hat Q_D^{LD}$ satisfy \eqref{eq:est}  and \eqref{eq:est2}, respectively, with
 \begin{align}
	&
\phi_*(p,N,s,\s,\d) \leq \frac{C}{\sqrt{\d}}\s\frac{p^{1/4}}{\sqrt{N}},\\
 &q_{\bt}(p,N,\s,\d) \leq \frac{C}{\sqrt{\d}} \Big(\s^2\frac{\sqrt{p}}{N} + \s\frac{\|\bt\|_2}{\sqrt{N}}\Big).
 \end{align}
Furthermore, for all $1\le s\leq p$, $\bt\in B_0(s)$, and $\kappa > 0$,
\begin{align}
q_{\bt}^*(p,N,\s,\d) &\leq  \min\Big(\frac{C}{\sqrt{\d}} \Big(\s^2\frac{\sqrt{p}}{N} + \s\frac{\|\bt\|_2}{\sqrt{N}}\Big), \|\bt\|_2^2\Big), \\
q_{s,\kappa}^*(p,N,\s,\d) &\leq  \min\Big(\frac{C}{\sqrt{\d}} \Big(\s^2\frac{\sqrt{p}}{N} + \s\frac{\kappa}{\sqrt{N}}\Big), \k^2\Big).
\end{align}
 \end{corollary}
Note that Corollary~\ref{cor:LDdense}  and Theorem~\ref{theo:plen_dense} are non-asymptotic in all parameters and valid for all $1\le s\leq p$. However, their upper bounds are minimax optimal only in the dense case $s> \sqrt{p}$. The matching lower bounds are given in  Theorems~\ref{th:lower1} and~\ref{th:lower2}, respectively. Taking them into account we obtain that,  under  the assumptions of Theorem~\ref{theo:plen_dense}, the estimators $\hat{\L}_D^{LD}$ and $\hat{Q}_D^{LD}$ are minimax optimal for $s> \sqrt{p}$.

\subsection{Estimation of $Q(\bt)$ and $\|\bt\|_2$ in the sparse case: $s \leq \sqrt{p}$}

In the sparse case $s\le \sqrt{p}$ of the low-dimensional regime, we use the estimators
\begin{equation}\label{def:QSLD}
	\hat Q_{S}^{LD} := \hat{Q}_{S} \Big(\hat{\bt}_{OLS},\hat{\bt}_{OLS},\hat{\s}_{OLS},(\XX_1^T\XX_1)^{-1},\a\Big), \quad \hat{\L}_{S}^{LD} := \Big|\hat Q_{S}^{LD}\Big|^{1/2},
\end{equation}
where $\hat Q_S$, $\hat{\bt}_{OLS}$ and $\hat{\s}_{ OLS}$ are defined in~\eqref{def:QS}, \eqref{eq:OLS} and~\eqref{eq:estsigOLS}, and $\a>0$ is a large enough constant. The following theorem holds.
\begin{theorem}\label{theo:plen:sparse}
Let $s\le \sqrt{p}$ and let Conditions~\ref{1}, \ref{2} and \ref{3} be satisfied. Then there exists a tuning parameter $\a>0$ depending only on $K,L$ and constants $C>0$, $\g\in (0,1)$ depending only on $K,L$ such that, for $p\le \min(\g n, n-14)$,
	\begin{equation}\label{eq:Q}
		\forall\bt\in B_0(s), \forall \s>0, \quad \esp_{\bt,\s} \big[\big(\hat Q_{S}^{LD}-Q(\bt)\big)^2\big]\le C\Big(\s^4\frac{s^2\log^2(1+\sqrt{p}/s)}{n^2}+\s^2\frac{\|\bt\|_2^2}{n}\Big)
	\end{equation}
	and
	\begin{equation}\label{eq:N}
		\forall\bt\in B_0(s), \forall \s>0, \quad \esp_{\bt,\s} \big[\big(\hat{\L}_S^{LD} - \|\bt\|_2\big) ^2\big]\le C\s^2\frac{s\log(1+\sqrt{p}/s)}{n}.
	\end{equation}
\end{theorem}

Theorem~\ref{theo:plen:sparse} immediately implies the following corollary analogous to Corollary~\ref{cor:LDdense}. 

\begin{corollary}\label{cor:LDsparse}
Let the assumptions of Theorem~\ref{theo:plen:sparse} hold. Then, for any $\d\in (0,1)$ the estimators  $\hat \L=\hat{\L}_S^{LD}$ and $\hat Q=\hat Q_S^{LD}$ satisfy \eqref{eq:est}  and \eqref{eq:est2}, respectively, with
 \begin{align}
	&
\phi_*(p,N,s,\s,\d) \leq \frac{C}{\sqrt{\d}}\s\sqrt{\frac{s\log(1+\sqrt{p}/s)}{N}},\\
 &q_{\bt}(p,N,\s,\d) \leq \frac{C}{\sqrt{\d}}\Big(\s^2\frac{s\log(1+\sqrt{p}/s)}{N}+\s\frac{\|\bt\|_2}{\sqrt{N}}\Big) .
 \end{align}
 Furthermore, for all $s\le \sqrt{p}$, $\bt\in B_0(s)$, and $\kappa > 0$,
\begin{align}
q_{\bt}^*(p,N,\s,\d) &\leq  \min\Big(\frac{C}{\sqrt{\d}} \Big(\s^2\frac{s\log(1+\sqrt{p}/s)}{N}+\s\frac{\|\bt\|_2}{\sqrt{N}}\Big), \|\bt\|_2^2\Big), \\
q_{s,\kappa}^*(p,N,\s,\d) &\leq  \min\Big(\frac{C}{\sqrt{\d}} \Big(\s^2\frac{s\log(1+\sqrt{p}/s)}{N}+\s\frac{\kappa}{\sqrt{N}}\Big), \k^2\Big).
\end{align}
  \end{corollary}
Corollary~\ref{cor:LDsparse} combined with Theorems~\ref{th:lower1} and~\ref{th:lower2} below implies that, under the assumptions of Theorem~\ref{theo:plen:sparse}, the estimators $\hat Q_{S}^{LD}$ and $\hat \L_{S}^{LD}$ are  minimax optimal when $s\leq \sqrt{p}$ in the sense defined in the Introduction.

\subsection{Signal detection}

Define the testing procedure 
\begin{align}
\Delta^{LD} &= \fcar \Big\{\hat{\L}^{LD} \geq \b\hat \s_{\sf srs}\s\sqrt{ \frac{s\log(1+\sqrt{p}/s)}{N}}\Big\},
\end{align}
where $\b>0$ is a constant, $\hat{\L}^{LD}=\hat{\L}_{S}^{LD}$ if $s \leq \sqrt{p}$, and $\hat{\L}^{LD}=\hat{\L}_{D}^{LD}$ if $s > \sqrt{p}$.
Theorems~\ref{theo:plen_dense},~\ref{theo:plen:sparse} and Lemma~\ref{lem:hatsigma_ols} in the appendix imply the following corollary.

\begin{corollary}\label{cor:test:lowdim}
	Let  Conditions~\ref{1}, \ref{2} and \ref{3} hold and $\d\in (0,1)$.
	Then there exist positive constants $\a,\b$, $\g\in (0,1)$ depending only on $K,d_4,L$ and a constant $C_\d>0$  depending only on $\d,K,d_4,L$ such that, for 
	 any $p\le \min(\g n, n-14)$, and any
	\begin{equation}
		\rho \geq C_\d\sqrt{\frac{s \log(1+\sqrt{p}/s)}{n}},
	\end{equation}
	we have
	\begin{equation}
		R(\D^{LD}, p,N,s, \rho) \leq \d,
	\end{equation}		
	 where $R(\cdot, p,N,s, \rho) $ is defined in~\eqref{eq:test}.
\end{corollary}
The proof of this corollary is straightforward (see, for example, the argument leading to Theorem 3 in \cite{carpentier2018minimax}). 
It follows that, for  $C_\d'>0$ depending only on $\d,K,d_4,L$,
\begin{equation}
	\rho_*(p,N,s,\d) \leq C_\d'\sqrt{\frac{s \log(1+\sqrt{p}/s)}{N}}.
\end{equation}
Moreover, Corollary~\ref{cor:test:lowdim} and Theorem~\ref{th:lower1} below imply that, under the assumptions of Corollary~\ref{cor:test:lowdim}, the test $\D^{LD}$ is minimax optimal in the sense defined in the Introduction. 

\section{Upper bounds  for the high-dimensional regime}\label{sec:pgen}

As the results of Section~\ref{sec:plen} are based on the preliminary least squares estimator, it is problematic to extend them to the zone $p>n$ where this estimator is not unique and cannot be written in the form \eqref{eq:OLS}.  In this section, we will use another preliminary estimator and provide analogs of Theorems \ref{theo:plen_dense} and 
\ref{theo:plen:sparse} in the high-dimensional regime that we define by the following condition: $p\ge \g n$, where $\g \in (0,1)$ is the maximal constant, for which Theorems~\ref{theo:plen_dense} and 
\ref{theo:plen:sparse} hold. Noteworthy, the theorems of this section are valid without the condition $p\ge \g n$.  However, to obtain their corollaries about the  
minimax optimal rates for all $n$ large enough that we state in this section we need a lower bound on $p$ as a function $n$. To make a connection to the results of Section~\ref{sec:plen}, we will use the lower bound of the form $p\ge \g n$ for some $\g>0$. Nevertheless, these corollaries also hold in a more general situation: $p\ge a_n$ for a given sequence $a_n$ that tends to $\infty$ with $n$.

\subsection{Preliminary estimators}\label{subsec:preliminaries}

 Instead of the least squares estimator, we now choose the preliminary estimator $\hat \bt$ as the Square-Root Slope estimator based on the first subsample $(\XX_1,\bY_1)$. This estimator is defined by the relation
\begin{equation}\label{definition_SRS}
	\hat{\bt}_{\sf srs} \in \arg\min_{\mathbf{t}\in \RR^p} \Big\{ \|\bY_1-\XX_1{\mathbf{t}}\|_2 +\|{\mathbf{t}}\|_{*}\Big\}.
\end{equation} 
Here, $\|\cdot\|_{*}$ denotes the sorted $\ell_1$-norm, that is,
\begin{equation}\label{sorted_norm}
	\|\mathbf{t}\|_{*} := \sum_{i=1}^d \l_i |t|_{(d-i+1)},
\end{equation}
where $|t|_{(1)}\le \cdots\le |t|_{(d)}$ are the order statistics of $|t_1|,\ldots,|t_d|$,  and 
\begin{equation}\label{lambdaj}
	\l_j = c_{\sf srs,1}\sqrt{\frac{\log(2p/j)}{n}}, \qquad j=1,\dots, p,	
\end{equation}
for some constant $c_{\sf srs,1}>0$ .	The next proposition follows by combining Corollary 6.2 in~\cite{Derumigny2017} and Theorem 8.3 in \cite{BellecLecueTsybakov2017}. 

\begin{proposition}\label{proposition_derumigny} 
	Let Conditions~\ref{1}, \ref{5} and \ref{6}  hold. Then there exist positive constants $c_{\sf srs,1},c_{\sf srs,2},c_{\sf srs,3}$ depending only on $L,M,\k_0$ such that
	\begin{align}
		\inf_{\bt \in B_0(s)} \inf_{\s>0}\ \prob_{\bt,\s} \bigg( &\|\hat{\bt}_{\sf srs}-\bt\|_2^2 \le c_{\sf srs,2} \s^2\frac{s}{n}\log(ep/s), \ \|\XX(\hat{\bt}_{\sf srs}-\bt)\|_2^2 \le c_{\sf srs,2} \s^2s\log(ep/s), \\
		& \qquad \|\hat{\bt}_{\sf srs}-\bt\|_{*}
		\le c_{\sf srs, 2} \s \frac{s}{n} \log(ep/s) \bigg) \\
	&	\ge 1-c_{\sf srs,3}\exp (-(n\wedge s\log(ep/s))/c_{\sf srs,3} ). \label{eq:derumigny}
	\end{align}
\end{proposition}
In the high-dimensional regime $p\ge \g n$ that we study in this section, we have $s\log(ep/s)\ge \log(ep)\ge \log(e\g n)$, so that  $$\exp (-(n\wedge s\log(ep/s))/c_{\sf srs,3} )\le C(\g n)^{-1/C}$$ for some constant $C>0$ depending only on $L,M,\k_0$. 

Next, we will use the following natural estimator of $\s$ based on the estimator $\hat{\bt}_{\sf srs}$:
\begin{equation}\label{definition_sigmaSRS}
	\hat\s_{\sf srs} = \frac{\|\bY_1-\XX_1\hat{\bt}_{\sf srs}\|_2}{\sqrt{n}}.
\end{equation}
The next proposition will be useful.
\begin{proposition}\label{prop_sigma_sqs}
	Let Conditions~\ref{1}, \ref{5} and \ref{6}  hold. Then there exist positive constants $c_{\sf srs,1},C,C'$ depending only on $L,M,\k_0$ such that, for any $t>0$, 
	\begin{align}\label{eq:sigma-sqs}
		\inf_{\bt \in B_0(s)} \inf_{\s>0} \ &\prob_{\bt,\s} \bigg(\bigg|\frac{\hat{\s}^2_{\sf srs}}{\s^2}-1\bigg|\le C\Big(\sqrt{\frac{t}{n}} +\frac{s}{n}\log(ep/s)\Big)
		\bigg) \\ 
		&\ge 1 - C'\exp (-(n\wedge t\wedge s\log(ep/s))/C' ).
	\end{align}
\end{proposition}
Setting here $t=\k_0^2n$ where $\k_0$ is the constant from Condition~\ref{6} and using the remark after Proposition \ref{proposition_derumigny}, we deduce from Proposition \ref{prop_sigma_sqs} that if Conditions~\ref{1}, \ref{5} and~\ref{6} hold and we are in the high-dimensional regime $p\ge \g n$ then there exist constants $C, C'>0$ depending only on $L,M,\k_0$ such that
\begin{equation}\label{sigmalasso}
	\inf_{\bt \in B_0(s)} \inf_{\s>0} \prob_{\bt,\s}\Big(|\hat{\s}^2_{\sf srs} -\s^2|\le \frac{\s^2}{2} \Big) \ge 1 - Ce^{-\frac{n\wedge s\log(ep/s)}{C}} \ge 1 -C'(\g n)^{-1/C'}.
\end{equation}

\subsection{Estimation of $Q(\bt)$ and $\|\bt\|_2$ in the dense case: $s > \sqrt{p}$}

Here, we assume that $N=2n$ and we divide the sample $(\XX,\bY)$ in two sub-samples  $(\XX_1,\bY_1)$ and $(\XX_2,\bY_2)$, each of size $n$. 
In the dense case $s > \sqrt{p}$ of the high-dimensional regime, we use the estimators
\begin{equation}\label{def_QDHD}
	\hat Q_{D}^{HD} := \hat Q_{D}(\hat{\bt}_{\sf srs}\big), \quad \hat{\L}_D^{HD} := \Big|\hat Q_{D}^{HD}\Big|^{1/2},
\end{equation}
where $\hat{Q}_D$ and $\hat{\bt}_{\sf srs}$ are defined in~\eqref{Qhatgeneral} and~\eqref{definition_SRS}. The risks of these estimators admit the following upper bounds.

\begin{theorem}\label{theo:pgen:dense}
	Let Conditions~\ref{1}, \ref{5} and~\ref{6} hold. Then there exists a tuning constant $c_{\sf srs,1}$ in the definition of $\hat{\bt}_{\sf srs}$ depending only on $L,M,\k_0$, and two constants $C,C'>0$ depending only on $L,M,\k_0$ such that  for any $0<v\le n^{1/3}$ we have
	\begin{equation}\label{eq1:theo:pgen:dense}
		\sup_{\bt \in B_0(s)} \sup_{\s>0} \prob_{\bt,\s}\Bigg(\Big[\hat Q_{D}^{HD}-Q(\bt)\Big]^2 > C v \Big(\s^4\frac{p}{n^2} +\s^2\frac{\|\bt\|_2^2}{n}\Big)\Bigg)\le C'\Big[e^{-v}+e^{-\frac{n\wedge s\log(ep/s)}{C'}}\Big]
	\end{equation}
	and
	\begin{equation}\label{eq2:theo:pgen:dense}
		\sup_{\bt \in B_0(s)} \sup_{\s>0} \prob_{\bt,\s}\Bigg(\Big|\hat{\L}_{D}^{HD}-\|\bt\|_2\Big|> C \s\frac{(pv)^{1/4}}{\sqrt{n}}\Bigg)\le C'\Big[e^{-v}+e^{-\frac{n\wedge s\log(ep/s)}{C'}}\Big].
	\end{equation}
\end{theorem} 
The next corollary follows immediately from Theorem~\ref{theo:pgen:dense} and the remark after Proposition~\ref{proposition_derumigny}. 
\begin{corollary}\label{cor:HDdense}
Let the assumptions of Theorem~\ref{theo:pgen:dense} hold. Let $p\ge \g n$ for some $\g>0$ and $\d\in (0,1)$. Then,
there exist $C_\d>0$ 
and an integer $N_0$, both depending only on $L,M,\k_0, \d,\g$, such that for $N\ge N_0$ the 
estimators  $\hat \L=\hat{\L}_D^{HD}$ and $\hat Q=\hat Q_D^{HD}$ satisfy \eqref{eq:est}  and \eqref{eq:est2}, respectively, with
 \begin{align}
	&
\phi_*(p,N,s,\s,\d) \leq C_\d\s\frac{p^{1/4}}{\sqrt{N}},\\
 &q_{\bt}(p,N,\s,\d) \leq C_\d \Big(\s^2\frac{\sqrt{p}}{N} + \s\frac{\|\bt\|_2}{\sqrt{N}}\Big).
 \end{align}
 Furthermore, for all $1\le s\leq p$, $\bt\in B_0(s)$, $\kappa > 0$  and $N\ge N_0$,
\begin{align}
q_{\bt}^*(p,N,\s,\d) &\leq  \min\Big(C_\d \Big(\s^2\frac{\sqrt{p}}{N} + \s\frac{\|\bt\|_2}{\sqrt{N}}\Big), \|\bt\|_2^2\Big), \\
q_{s,\kappa}^*(p,N,\s,\d) &\leq  \min\Big(C_\d \Big(\s^2\frac{\sqrt{p}}{N} + \s\frac{\kappa}{\sqrt{N}}\Big), \k^2\Big).
\end{align}
 \end{corollary}
It follows from Corollary~\ref{cor:HDdense} and Theorems~\ref{th:lower1} and~\ref{th:lower2} below that, under the assumptions of Theorem~\ref{theo:pgen:dense}, the above bounds are minimax optimal when $s> \sqrt{p}$ in the sense defined in the Introduction.

{\bf Remark 1.} If Condition ($P$) is not satisfied, that is  we do not have $s\log(ep/s)< \k_0 N$ for $\kappa_0>0$ small enough, 
there still exists an estimator of $\|\theta\|_2^2$ that is consistent whenever $p=o(N^2)$. 
Indeed, consider the estimator $\hat Q_D(\mathbf 0)$, \ie the generic estimator of $\|\theta\|_2^2$ defined in~\eqref{Qhatgeneral} applied to the null vector. It follows immediately from Theorem~\ref{theo:general} that for any $\delta\in (0,1)$ the estimator $\hat Q_D(\mathbf 0)$ satisfies~\eqref{eq:est2} with
\begin{equation}\label{eq:q2test}	
q_{\bt}(p,N,\s,\d) \leq C_\d \Big[\frac{\|\bt\|_2}{\sqrt{N}}(\|\bt\|_2+\s)+\frac{\sqrt{p}}{N}(\s^2+\|\bt\|_2^2)\Big],
\end{equation}
where $C_\d>0$ depends only on $\delta$ and $d_4$. 

\subsection{Estimation of $Q(\bt)$ and $\|\bt\|_2$ in the sparse case: $s \leq \sqrt{p}$}

Here, we assume that $N=3n$ and we divide the sample $(\XX,\bY)$ in three sub-samples  $(\XX_1,\bY_1), (\XX_2,\bY_2)$ and $(\XX_3,\bY_3)$, each of size $n$. 
In the sparse case $s \leq \sqrt{p}$ of the high-dimensional regime, we use the estimators
\begin{equation}\label{def_QSHD}
	\hat{Q}_{S}^{HD} := \hat{Q}_{S}(\hat{\bt}_{\sf srs},\tilde{\bt}_{\sf srs}, \hat{\s}_{\sf srs}, 2\mathbb I_p, \a), \quad \hat{\L}_S^{HD} := \Big|\hat{Q}_{S}^{HD}\Big|^{1/2},
\end{equation}
where the estimators $\hat{\bt}_{\sf srs}$ and $\hat{\s}_{\sf srs}$ are defined in~\eqref{definition_SRS} and~\eqref{definition_sigmaSRS}, $\a>0$ is a constant large enough and $\tilde{\bt}_{\sf srs}$ is a debiased estimator derived from $\hat \bt_{\sf srs}$:
\begin{equation}
	\tilde{\bt}_{\sf srs} := \hat{\bt}_{\sf srs}+\frac{1}{n}\XX_3^T(\bY_3-\XX_3\hat{\bt}_{\sf srs}).
\end{equation}
Note that the triplet $(\hat{\bt}_{\sf srs},\tilde{\bt}_{\sf srs},\hat{\s}_{\sf srs})$ is independent of $(\XX_2,\bY_2)$. The following theorem gives bounds on the rates of convergence of estimators $	\hat{Q}_{S}^{HD} $ and 	$\hat{\L}_{S}^{HD}$.
\begin{theorem}\label{theo:pgen:sparse12}
	Let $s \le \sqrt{p}$ and let Conditions~\ref{1}, \ref{5} and~\ref{6} hold. 
	Then there exist a tuning constant $c_{\sf srs,1}$ in the definition of $\hat{\bt}_{\sf srs}$ depending only on $L,M,\k_0$, and two positive constants $C,C'$ depending only on $L,M,\k_0$ such that 
	\begin{equation}
		\sup_{\bt \in B_0(s)} \sup_{\s>0} \prob_{\bt,\s}\Big[\big(\hat{Q}_S^{HD}-Q(\bt)\big)^2 \le C\Big(  \s^4\frac{s^2\log^2(1+\sqrt{p}/s)}{n^2} +\s^2\frac{\|\bt\|_2^2}{n}\Big)\Big] \ge 1-C'e^{-\frac{n\wedge s\log(ep/s)}{C'}},
	\end{equation}
	and
	\begin{equation}
		\sup_{\bt \in B_0(s)} \sup_{\s>0} \prob_{\bt,\s}\Big[\big(\hat{\L}_S^{HD}-\|\bt\|_2\big)^2 \le C\s^2 \frac{s\log(1+\sqrt{p}/s)}{n} \Big] \ge 1-C'e^{-\frac{n\wedge s\log(ep/s)}{C'}}.
	\end{equation}
\end{theorem}
The next corollary follows immediately from Theorem~\ref{theo:pgen:sparse12} and the remark after Proposition~\ref{proposition_derumigny}. 
\begin{corollary}\label{cor:HDsparse}
Let the assumptions of Theorem~\ref{theo:pgen:sparse12} hold. 
Let $p\ge \g n$ for some $\g>0$ and $\d\in (0,1)$. Then,
there exist $C_\d>0$ 
and an integer $N_0$, both depending only on $L,M,\k_0, \d,\g$, such that for $N\ge N_0$ the 
estimators  $\hat \L=\hat{\L}_S^{HD}$ and $\hat Q=\hat Q_S^{HD}$ satisfy \eqref{eq:est}  and \eqref{eq:est2}, respectively, with
 \begin{align}
	&
\phi_*(p,N,s,\s,\d) \leq C_\d\s\sqrt{\frac{s\log(1+\sqrt{p}/s)}{N}},\\
 &q_{\bt}(p,N,\s,\d) \leq C_\d\Big(\s^2\frac{s\log(1+\sqrt{p}/s)}{N}+\s\frac{\|\bt\|_2}{\sqrt{N}}\Big) .
 \end{align}
 Furthermore, for all $s\le \sqrt{p}$, $\bt\in B_0(s)$, $\kappa > 0$ and $N\ge N_0$,
\begin{align}
q_{\bt}^*(p,N,\s,\d) &\leq  \min\Big(C_\d\Big(\s^2\frac{s\log(1+\sqrt{p}/s)}{N}+\s\frac{\|\bt\|_2}{\sqrt{N}}\Big), \|\bt\|_2^2\Big), \\
q_{s,\kappa}^*(p,N,\s,\d) &\leq  \min\Big(C_\d \Big(\s^2\frac{s\log(1+\sqrt{p}/s)}{N}+\s\frac{\kappa}{\sqrt{N}}\Big), \k^2\Big).
\end{align}
 \end{corollary}
It follows from Corollary~\ref{cor:HDsparse} and Theorems~\ref{th:lower1} and~\ref{th:lower2} below that, under the assumptions of Theorem~\ref{theo:pgen:sparse12}, the above bounds are minimax optimal when $s\le \sqrt{p}$ in the sense defined in the Introduction.

\subsection{Signal detection}
  
Define the testing procedure
\begin{align}
	\Delta^{HD} &= \fcar \Big\{\hat{\L}^{HD} \geq \b\hat \s_{\sf srs}\s\sqrt{ \frac{s\log(1+\sqrt{p}/s)}{N}}\Big\},
\end{align}
where $\b>0$ is a constant,  $\hat{\L}^{HD}=\hat{\L}_{S}^{HD}$ if $s \leq \sqrt{p}$, and $\hat{\L}^{HD}=\hat{\L}_{D}^{HD}$ if $s > \sqrt{p}$. Theorems~\ref{theo:pgen:dense},~\ref{theo:pgen:sparse12} and Proposition~\ref{prop_sigma_sqs} imply the following corollary.

\begin{corollary}\label{cor:test:highdim}
		Let  Conditions~\ref{1}, \ref{5} and (P) hold and $\d\in (0,1)$. Let $p\ge \g n$ for some $\g>0$.
	Then there exist positive constants $\a,\b,C_\d, N_0$ depending only on $\d,L,M,\k_0, \g$ such that, for 
	any $N\geq N_0$ and any
	\begin{equation}
	\rho \geq C_\d\sqrt{\frac{s \log(1+\sqrt{p}/s)}{N}},
	\end{equation}
	we have
	\begin{equation}
	R(\D^{HD}, p,N,s, \rho) \leq \d,
	\end{equation}		
	where $R(\cdot, p,N,s, \rho) $ is defined in~\eqref{eq:test}.
\end{corollary}
The proof of this corollary is straightforward (see, for example, the argument leading to Theorem 3 in \cite{carpentier2018minimax}). 
It follows that, for  $C_\d'>0, N_0>0$ depending only on $\d,L,M,\k_0, \g$ and any $N\geq N_0$,
\begin{equation}
\rho_*(p,N,s,\d) \leq C_\d'\sqrt{\frac{s \log(1+\sqrt{p}/s)}{N}}.
\end{equation}
Moreover, Corollary~\ref{cor:test:highdim} and Theorem~\ref{th:lower1} below imply that, under the assumptions of Corollary~\ref{cor:test:highdim}, the test $\D^{HD}$ is minimax optimal in the sense defined in the Introduction.

\section{Lower bounds}\label{sec:lower}

\begin{theorem}\label{th:lower1}
Assume that $\mathbb X$ is a matrix with i.i.d.~standard normal entries and $\bxi$ is an i.i.d.~standard normal noise. 
For any $\d\in (0,1)$, there exists $c_\d>0$ depending only on $\d$ such that 
\begin{equation}\label{th:lower1_1}
\inf_{ \Delta} R(\Delta,p,N,s,\rho)\ge \d,
\end{equation}
for 
\begin{equation}\label{th:lower1_2}
\rho \le \rho(p,N,s,\d): = c_\d\min\left( \sqrt{\frac{s\log(1+\sqrt{p}/s)}{N}},1\right),
\end{equation}
where $R(\bar\Delta,p,N,s,\rho)$ is defined in \eqref{eq:test}, and $\inf_{ \Delta}$ denotes the infimum over all tests $\Delta$. As a consequence, for all $\d\in (0,1)$ we have 
\begin{equation}\label{th:lower1_3}
\rho^*(p,N,s, \delta) \ge c_\d\min\left( \sqrt{\frac{s\log(1+\sqrt{p}/s)}{N}},1\right),
\end{equation}
and 
\begin{equation}\label{th:lower1_4}
\phi^*(p,N,s,\s, \delta) \ge \frac{c_\d}{2}\s\min\left( \sqrt{\frac{s\log(1+\sqrt{p}/s)}{N}},1\right).
\end{equation}
\end{theorem}

Note that Theorem~\ref{th:lower1} can be obtained by combining \cite[Lemma 17]{carpentier-verzelen2019a} and the proof of \cite[Proposition 6]{carpentier-verzelen2019a} if we take into account that \cite[Proposition 6]{carpentier-verzelen2019a} is stated with the maximum over the covariance matrices of the design while its proof uses only the identity matrix. For completeness, we provide a direct proof of Theorem~\ref{th:lower1}
in Section \ref{sec:proof:lower1} below. 

\begin{theorem}\label{th:lower2}
Assume that $\mathbb X$ is a matrix with i.i.d.~standard normal entries and $\bxi$ is and i.i.d.~standard normal noise. 
Then for any $\d\in (0,1)$, there exists $c_\d>0$ that depends only on $\delta$ such that 
$$
\inf_{\hat T} \sup_{\substack{\bt\in B_0(s): \\ \|\bt\|_2\le\kappa}} \ \sup_{\s>0} \prob_{\bt,\s}\Big(|\hat T - Q(\bt)|\ge c_\d \bar q_{s,\kappa}(p,N,  \s)\Big)\ge \d,
$$ 
where $\inf_{\hat T}$ denotes the infimum over all estimators, and
$$
\bar  q_{s,\kappa}(p,N,  \s) =  \min\Big(   \s^2 \min \Big(\frac{ s\log(1+\sqrt{p}/s)}{N},1\Big) + \frac{\s \kappa}{\sqrt{N}}, \ \kappa^2\Big).
$$
\end{theorem}
This theorem immediately implies the following corollary.
\begin{corollary}
Under the assumptions of Theorem~\ref{th:lower2} we have
$$
q_{\bt}^*(p,N, \s,\delta) \geq  c_\d\min\Big(   \s^2 \min \Big(\frac{ \|\bt\|_0\log(1+\sqrt{p}/\|\bt\|_0)}{N},1\Big) + \frac{\s \|\bt\|_2}{\sqrt{N}}, \ \|\bt\|_2^2\Big),
$$
and
$$
q_{s,\kappa}^*(p,N,  \s,\delta) \geq c_\d  \min\Big(   \s^2 \min \Big(\frac{ s\log(1+\sqrt{p}/s)}{N},1\Big) + \frac{\s \kappa}{\sqrt{N}}, \ \kappa^2\Big).
$$
\end{corollary}


\section{Proof of Theorem~\ref{theo:general}}

In this proof, we write for brevity $(\XX,\bY)$ instead of $(\XX_2,\bY_2)$. Using the notation $\bu=\hat{\bt}-\bt$, and $\bw_j=\sum_{k\neq j} u_k\bX_{k}$
where $u_k$'s are the components of $\bu$ we have  
\begin{align}\label{theogen:decomp}
	&\hat{Q}_D = Q(\bt) + \sum_{j=1}^p \Big[ 2\t_ju_j\Big(1-\frac{\|\bX_{j}\|_2^2}{n} \Big)+2\t_j \frac{\bX_{j}^T}{n}(\s\bxi-\bw_j)\Big] \\
	&+\sum_{j=1}^p \Big[u_j^2+\frac{2}{n}u_j\bX_{j}^T(\s\bxi-\XX\bu)+\frac{1}{n(n-1)}\sum_{k\neq l} X_{kj}X_{lj}(\s\xi_{k}-\bX^k \bu)(\s\xi_{l}-\bX^l \bu)\Big].
\end{align}
Note that $\esp_{\bt,\s} (\hat{Q}_D|\hat{\bt})=Q(\bt)$. Therefore, we need to evaluate the conditional variance of $\hat{Q}_D$ given $\hat{\bt}$. We have
\begin{align}\label{varcrossterm1}
	\var_{\bt,\s} \Big(\sum_{j=1}^p \Big[ 2\t_ju_j\Big(1-\frac{\|\bX_{j}\|_2^2}{n} \Big)+2\s\t_j \frac{\bX_{j}^T}{n}\bxi\Big] \suchthat \hat{\bt} \Big)&\le \frac{C}{n}\sum_{j=1}^p \Big[ \t_j^2u_j^2+\s^2\t_j^2 \Big]\\
	&\le \frac{C}{n}\|\bt\|_2^2(\|\bu\|_2^2+\s^2),
\end{align}
and, since  $\esp[(\bX_{j}^T\bX_{l})^2] = \esp \sum_{m=1}^{n}(X_{jm}X_{lm})^2 =n $ for $j\ne l$,
\begin{align}
	\var_{\bt,\s}\Big(\sum_{j=1}^p\t_j \frac{\bX_{j}^T}{n}\bw_j \suchthat \hat{\bt} \Big) 
	&= \sum_{j=1}^p \frac{\t_j^2}{n^2}\esp_{\bt,\s}\Big[\Big(\sum_{k\neq j}u_k\bX_{j}^T\bX_{k}\Big)^2\suchthat \hat{\bt}\Big]
	+\sum_{j\neq l}\frac{\t_j\t_l}{n^2}\esp_{\bt,\s}(\bX_{j}^T\bw_j\bX_{l}^T\bw_l\suchthat \hat{\bt}) \hspace{.7cm} \nonumber
	\\
	&=\sum_{j\neq k} \frac{\t_j^2}{n}u_k^2 + \sum_{j\neq l}\frac{\t_j\t_l}{n}u_ju_l \phantom{\sum_{j=1}^p \frac{\t_j^2}{n^2}} 
	\\
	&\le \frac2n \|\bt\|_2^2\|\bu\|_2^2.\label{varcrossterm2} \phantom{\sum_{j=1}^p \frac{\t_j^2}{n^2}}
\end{align}
Finally, the double sum in~\eqref{theogen:decomp} can be written as $U=\frac{2}{n(n-1)}\sum_{i < j}h\big((\bX^{i},\xi_{i}),(\bX^{j},\xi_{j})\big)$ where
\begin{equation}
	h\big((\bX^{i},\xi_{i}),(\bX^{j},\xi_{j})\big) = \sum_{k=1}^p \big[X_{ik}(\s \xi_{i}-\mathbf{X}^i\bu)+u_k\big]\big[X_{jk}(\s \xi_j-\mathbf{X}^j\bu)+u_k\big].
\end{equation}
Since,conditionally on $\hat{\bt}$, the terms in $U$ are zero-mean and uncorrelated, the variance of $U$ conditionally on $\hat{\bt}$ equals
\begin{equation}
	\var_{\bt,\s} (U\suchthat \hat{\bt}) = \frac{2}{n(n-1)} \var_{\bt,\s} \Big[ h\big((\bX^{1},\xi_{1}),(\bX^{2},\xi_{2})\big) | \hat{\bt}\Big].
\end{equation}
Now, for all $i,j$,
\begin{equation}\label{comp:ustat1}
	\esp_{\bt,\s}  \Big[\big(X_{ij}(\s\xi_{i}-\bX^i\bu)+u_j\big)^2\suchthat \hat \bt \Big] \le C(\|\bu\|_2^2+\s^2)
\end{equation}
and, for $j\neq k$,
\begin{equation}\label{comp:ustat2}
	\esp_{\bt,\s} \Big[\big(X_{ij}(\s\xi_i-\bX^i\bu)+u_j\big)\big(X_{ik}(\s\xi_i - \bX^i\bu)+u_k\big) \suchthat \hat \bt\Big]=u_ju_k.
\end{equation}
Hence, the conditional variance of $U$ given $\hat{\bt}$ satisfies
\begin{equation}\label{var:ustat}
		\var_{\bt,\s} (U\suchthat \hat{\bt})  \le C\frac{p}{n^2}(\|\bu\|_2^2+\s^2)^2.
\end{equation}
The result follows by combining~\eqref{varcrossterm1}, \eqref{varcrossterm2} and~\eqref{var:ustat}.

\section{Proofs for Section~\ref{sec:plen}}

\subsection{Preliminary lemmas}

The following lemma will be useful.
\begin{lemma}\label{lem:Yuhao}
	Let $p\le \min(\g n, n-14)$ for some $\g\in (0,1)$ small enough. If Condition~\ref{2} holds then there exists a constant $C>0$ depending only on $K$ such that
	\begin{equation}
	\esp \big(\lambda^{-4}_{\min}(\XX_1^T\XX_1)\big)\le Cn^{-4}.
	\end{equation}
	If Conditions~\ref{4} and~\ref{2} hold then there exists a constant $C>0$ depending only on $K,c_8$ such that
	\begin{equation}
	\esp_{\bt,\s}\|\hat{\bt}_{OLS} - \bt\|_2^8\le C\s^8.
	\end{equation}
\end{lemma} 
\begin{proof}
In this proof, we write for brevity $(\XX,\bY)$ instead of $(\XX_1,\bY_1)$. First, we note that, almost surely, 
\begin{equation}~\label{dec:lemY0}
	\|\hat{\bt}_{OLS}-\bt\|_2^8\le \s^8 \lambda^{4}_{\max}(\XX(\XX^T\XX)^{-2}\XX^T) \|\bxi\|_2^8 = \s^8 \lambda^{-4}_{\min}(\XX^T\XX) \|\bxi\|_2^8.
\end{equation}
Thus, the proof of the lemma will be complete if we show that $\esp \big[ \lambda_{\min}^{-4}(\XX^T\XX) \big] \le Cn^{-4}$, where $C>0$ is a constant depending only on $K$. To this end, we use the following proposition. 
\begin{proposition}\label{prop:LecueMendelson}
	If Condition~\ref{2} holds and $p\le \g n$ for some $\g\in(0,1)$ small enough then, with probability at least $1-\exp(-Cn)$,
	\begin{equation}
		\l_{\min}(\XX^T\XX) > C'n,
	\end{equation}	
	where $C,C'$ are positive constants depending only on $K$.
\end{proposition}
This proposition follows immediately from combining item~(1) of~Corollary~2.5 in~\cite{LecueMendelson2017} with the remark on page~884  in~\cite{LecueMendelson2017} that gives the explicit form of the small ball condition for a random vector whose coordinates are
independent random variables with bounded Lebesgue density.

Using Proposition~\ref{prop:LecueMendelson} we get
\begin{equation}\label{dec:lemY}
	\esp \big[ \lambda_{\min}^{-4}(\XX^T\XX) \big] \le (C'n)^{-4} + \sqrt{\esp \big[\lambda_{\min}^{-8}(\XX^T\XX) \big]} e^{-Cn/2},
\end{equation}
where $C,C'$ are the constants from Proposition~\ref{prop:LecueMendelson}.

To complete the proof, we now show that   $\esp \big[\lambda_{\min}^{-8}(\XX^T\XX) \big]\le Cn^8$, where $C>0$ is a constant depending only on $K$. 
Denote by $\XX_{-1}$ the matrix obtained by removing the first column from $\XX$ and by  ${\rm Span}(\XX_{-1})$ the span of its columns.  Arguing quite analogously to the proof of Lemma~3 in~\cite{carpentier2018minimax} we obtain 
\begin{equation}
	\esp \big[ \lambda_{\min}^{-8}(\XX^T\XX)\big] \leq n^8 \esp[{\rm dist}(\bX_1,{\rm Span}(\XX_{-1}))^{-16}],
\end{equation}
where ${\rm dist}(\bX_1,{\rm Span}(\XX_{-1}))$ is the Euclidean distance from $\bX_1$ to ${\rm Span}(\XX_{-1})$.
Note that, under Condition~\ref{2}, the vector space ${\rm Span}(\XX_{-1})$ has  almost surely dimension $p-1$. Recall also that $\bX_1$ is independent from $\XX_{-1}$. Thus,
\begin{equation}
	\esp[{\rm dist}(\bX_1,{\rm Span}(\XX_{-1}))^{-16}] \le \max_S \esp[{\rm dist}(\bX_1,S)^{-16}] = \max_S \esp[\|P_S^\perp \bX_1\|_2^{-16}] ,
\end{equation}
where the maximum is taken over all $(p-1)$-dimensional vector subspaces $S$ of $\RR^n$ and we denote by
$P_S^\perp$ the orthogonal projector on $S^\perp$. Furthermore, by Theorem~1.1 in~\cite{RudelsonVershynin2015}, under Condition~\ref{2} the density of $P_S^\perp \bX_1$ is bounded almost everywhere by $C_0^{n-p+1}$, where $C_0=CK$ for an absolute constant $C>0$. 
 We have
\begin{equation}
	\esp[\|P_S^\perp \bX_1\|_2^{-16}] \le  C_0^{16} + \esp \big[ \|P_S^\perp \bX_1\|_2^{-16} \fcar_{\|P_S^\perp \bX_1\|_2 < C_0^{-1} } \big].
\end{equation}
As
 $\max_{k\ge 1}2\pi^{(k+1)/2}/  \G\big((k+1)/2\big)$ is less than an absolute constant, it holds that, for $n\ge p+14$,
\begin{align}
	\esp \big[ \|P_S \bX_1\|_2^{-16} \fcar_{\|P_S \bX_1\|_2 < C_0^{-1}} \big] &\le C_0^{n-p+1} \int_{\|z\|_2\le C_0^{-1}} \|z\|^{-16} \, {\rm d}z 
	\\
	&= C_0^{n-p+1} \frac{2\pi^{(n-p+2)/2}}{\G\big((n-p+2)/2\big)}\int_0^{C_0^{-1}} r^{n-p-15} \, {\rm d}r\le C
\end{align}
where $C>0$ is a constant depending only on $K$. Hence, $\esp \big[ \lambda_{\min}^{-8}(\XX^T\XX)\big] \leq Cn^8$ where $C>0$ depends only on $K$. Together with~\eqref{dec:lemY0} and~\eqref{dec:lemY} this proves the lemma.
\end{proof}

The following lemma summarizes some properties of the estimator 
$$
\hat{\s}_{ OLS} := \frac{\|\bY_1-\XX_1\hat{\bt}_{OLS}\|_2}{\sqrt{n-p}}.
$$
\begin{lemma}\label{lem:hatsigma_ols}
	If Condition~\ref{1} holds and $p\le \g n$ for some $\g\in(0,1)$, then there exists a constant $C>0$ depending only on $L$ such that
	\begin{equation}
	\prob_{\bt,\s}\Big( \Big|\frac{\hat{\s}^2_{OLS}}{\s^2}- 1\Big|\le \frac{1}{2}\Big) \ge 1-2e^{-(1-\g)n/C} \quad \text{and} \quad \esp_{\bt,\s} \big(\hat \s_{OLS}^4\big) \le C \s^4.
	\end{equation}
\end{lemma}

\begin{proof} 
We have
\begin{equation}
	\frac{\hat{\s}_{OLS}^2}{\s^2}=\frac{\| \mathbb{P} \bxi\|_2^2}{n-p},
\end{equation}
where $\mathbb{P}$ is the orthogonal projector on the orthogonal complement of the range of $\XX_1$, which has rank $n-p$ almost surely 
and satisfies $\|\mathbb{P}\|_{\infty}=1$, $\|\mathbb{P}\|_F^2=n-p$. Applying the Hanson-Wright inequality for subGaussian variables (\cf Theorem 6.2.1 in \cite{vershyninbook}) conditionally on $\XX_1$  we obtain that there is an absolute constant $C>0$ such that, for all $t\ge 0$, 
\begin{equation}\label{hanson}
	\prob_{\bt,\s}\Big(\Big| \frac{\hat{\s}_{OLS}^2}{\s^2}-1\Big| >t\Big)\le 2\exp \Big(-C\min\Big\{\frac{(n-p)t^2}{L^4}, \frac{(n-p)t}{L^2}\Big\} \Big).
\end{equation}
Hence for $0\le u\le n-p$, 
\begin{equation}\label{lemhatsigmaols1}
	\prob_{\bt,\s}\Big(\Big| \frac{\hat{\s}_{OLS}^2}{\s^2}-1\Big| >L^2\sqrt{\frac{u}{n-p}}\Big)\le 2\exp(-Cu).   
\end{equation} 
Taking $u=(2\max (L^2,1))^{-1}(n-p) \ge C'(1-\g)n$ yields the first inequality of the lemma. The bound on the fourth moment of $\hat \s_{OLS}$ follows by integrating \eqref{hanson}.
\end{proof}

	Assuming that Condition~\ref{2} holds we set
	\begin{equation}
		\htau_j=2\a\hat{\s}_{OLS} \sqrt{[(\XX_1^T\XX_1)^{-1}]_{jj} \log(1+p/s^2)},
	\end{equation}
	where $[(\XX_1^T\XX_1)^{-1}]_{jj}$ denotes the $j$th diagonal entry of matrix $(\XX_1^T\XX_1)^{-1}$ \footnote{Recall that $\XX_1^T\XX_1$ is almost surely invertible under Condition~\ref{2}, so that operating with $(\XX_1^T\XX_1)^{-1}$ is formally legitimate everywhere except an event of zero probability, on which we set, for example, $\htau_j=0$. We do not further invoke this detail since it does not influence the argument.}.
	\begin{lemma}\label{lem:indic}
	Let $p\le\gamma n$ for some $\g\in(0,1)$ and let Conditions~\ref{2} and~\ref{1} hold. Then there exist $\a>0$ is large enough and a constant $C>0$ depending only on $\g,L$ such that 
	\begin{equation}
		\sup_{\bt\in B_0(s)} \sup_{j: \t_j=0} \prob_{\bt,\s}\big(|\hat{\t}_{OLS,j}|>\htau_j\big)\le  C\frac{s^4}{p^2}
	\end{equation}
	and, for all $j\in \{1,\ldots, p\}$,
	\begin{equation}
		\sup_{\bt\in B_0(s)} \prob_{\bt,\s}\big(|\hat{\t}_{OLS,j}|\le \htau_j, |\t_j|\ge 2\htau_j  | \XX_1\big)\le 2 \exp\Big(-\frac{\t_j^2}{C\s^2 [(\XX_1^T\XX_1)^{-1}]_{jj}}\Big).
	\end{equation}	 
\end{lemma}

\begin{proof}
	For brevity, in this proof we write $(\XX,\bY)$, $\hat{\bt}$ and $\hat{\s}$ instead of $(\XX_1,\bY_1), \hat{\bt}_{OLS}$ and $\hat{\s}_{OLS}$,  respectively.  Set 
	\begin{equation}
		\tau_j=\a \s \sqrt{[(\XX^T\XX)^{-1}]_{jj} \log(1+p/s^2)}.
	\end{equation}
	Using Lemma~\ref{lem:hatsigma_ols}, we have 
	\begin{equation}
		\prob_{\bt,\s}\big(|\hat{\t}_j|>\htau_j\big)\le \prob_{\bt,\s}\big(|\hat{\t}_j|> \tau_j\big) + 2 e^{-C'(1-\g)n}.
	\end{equation}
	Next, if $\t_j=0$, then the fact that $\hat{\bt}=\bt+\s(\XX^T\XX)^{-1}\XX^T \bxi$ yields
	\begin{equation}
		\prob_{\bt,\s}\big(|\hat{\t}_{j}|>\tau_j \suchthat \XX \big) = \prob\big\{ |[(\XX^T\XX)^{-1}\XX^T\bxi]_j|>\frac{\tau_j}{\s} \suchthat \XX \big\}.
	\end{equation}
	Note that, conditionally on $\XX$, the random variable $[(\XX^T\XX)^{-1}\XX^T\bxi]_j$ is $L\sqrt{[(\XX^T\XX)^{-1}]_{jj}}$-subGaussian. Thus, if $\a>2L$, 
	\begin{equation}
		\prob\big\{|[(\XX^T\XX)^{-1}\XX^T\bxi]_j|>\frac{\tau_j}{\s} \suchthat \XX \big\} \le 2 \exp\Big(-\frac{\tau_j^2}{2\s^2 L^2 [(\XX^T\XX)^{-1}]_{jj}}\Big)
		\le 
		C\frac{s^4}{p^2}.
	\end{equation}
	Combining the above bounds and taking into account the condition $p\le \g n$ we obtain the first inequality of the lemma.  To prove the second inequality, we note that
	\begin{align}
	 \prob_{\bt,\s}\big(|\hat{\t}_{j}|\le \htau_j, |\t_j|\ge 2\htau_j  | \XX\big)&\le
	 \prob_{\bt,\s}\big(|\hat{\t}_{j}-\t_j|\ge |\t_j|/2  | \XX\big)\\
	 &
	 = \prob\big\{|[(\XX^T\XX)^{-1}\XX^T\bxi]_j|> |\t_j|/(2\s) \suchthat \XX \big\} \\
	 &
	  \le 2 \exp\Big(-\frac{\t_j^2}{8L^2\s^2 [(\XX^T\XX)^{-1}]_{jj}}\Big).
	\end{align}	
\end{proof}

\subsection{Proof of Theorem~\ref{theo:plen_dense}}

The first inequality of the theorem is straightforward in view of Theorem~\ref{theo:general} and Lemma~\ref{lem:Yuhao}. The second inequality (that is, the upper bound on the squared risk of $\hat{\L}_D^{LD}$) is deduced from the first inequality exactly in the same way as it is done in the proof of Theorem 8 in~\cite{Collier2017}. 

\subsection{Proof of Theorem~\ref{theo:plen:sparse}}\label{sec:proofTH3}

For brevity, in this subsection  we write $(\XX,\bY)$, $\hat{\bt}$ and $\hat{\s}$ instead of $(\XX_2,\bY_2), \hat{\bt}_{OLS}$ and $\hat{\s}_{OLS}$,  respectively. We also set $\hat Q = \hat Q_{S}^{LD}$, where $\hat{Q}_{S}^{LD}$ is defined in~\eqref{def:QSLD}. Let $\htau_j$ and $\tau_j$ be the same as in~Lemma~\ref{lem:indic}.

Denoting by $\mathcal{S}$ the support of $\bt$ and by $S^\complement$ its complement, we have
\begin{align}\label{theo:plen_sparse_1}
	\esp_{\bt,\s}\big(\hat{Q}-Q(\bt)\big)^2\le  2\esp_{\bt,\s}\Big(\sum_{j\in \mathcal S^\complement} a_j(\hat \bt)\fcar_{|\hat{\bt}_j|>\hat{\tau}_j}\Big)^2 + 2\esp_{\bt,\s}\Big(\sum_{j\in \mathcal S} \{ a_j(\hat \bt) \fcar_{|\hat{\bt}_j|>\hat{\tau}_j}-\t_j^2 \} \Big)^2.
\end{align}
Consider the first sum on the right-hand side of \eqref{theo:plen_sparse_1}. Acting as in the proof of Theorem~\ref{theo:general} we get
\begin{equation}\label{eq:star}
	\sum_{j\in \mathcal S^\complement} a_j(\hat \bt) \fcar_{|\hat{\t}_j|>\hat{\tau}_j} = \frac{2}{n(n-1)}\sum_{i< j} \tilde h\big((\bX^{i},\xi_{i}),(\bX^{j},\xi_{j})\big),
\end{equation}
where
\begin{equation}
	\tilde{h} \big((\bX^{i},\xi_{i}),(\bX^{j},\xi_{j})\big) = \sum_{k \in \mathcal S^\complement} \big[X_{ik}(\s\xi_i - \bX^i\bu) + u_k\big]\big[X_{jk}(\s\xi_j - \bX^j\bu) + u_k\big]\fcar_{|\hat{\t}_k|>\hat{\tau}_k}.
\end{equation}
Conditionally on $(\XX_1,\bY_1)$, the terms in \eqref{eq:star} are zero-mean and distinct terms are uncorrelated. 
Thus, 
\begin{equation}\label{for:th:51}
	\esp_{\bt,\s} \Big[\Big( \sum_{j\in \mathcal S^\complement} a_j(\hat \bt)\fcar_{|\hat{\t}_j|>\hat{\tau}_j}\Big)^2 \suchthat \XX_1,\bY_1\Big]
	= \frac{2}{n(n-1)} \esp_{\bt,\s} \big[\tilde{h}^2\big((\bX^{1},\xi_{1}),(\bX^{2},\xi_{2})\big)\suchthat \XX_1,\bY_1\big]
\end{equation}
and, by an argument analogous to~\eqref{comp:ustat1} and~\eqref{comp:ustat2}, we find
\begin{align}\label{for:th:5}
&\esp_{\bt,\s} \Big[\Big( \sum_{j\in \mathcal S^\complement} a_j(\hat \bt)\fcar_{|\hat{\t}_j|>\hat{\tau}_j}\Big)^2 \Big]\\
&\qquad \le
	\frac{C}{n^2}\esp_{\bt,\s} \Big[(\|\bu\|_2^2+\s^2)^2 \sum_{k\in \mathcal S^\complement} \fcar_{|\hat{\t}_k|>\hat{\tau}_k} + \sum_{k,k'\in \mathcal S^\complement: k'\neq k} u_{k'}^2u_k^2 \fcar_{|\hat{\t}_{k'}|>\hat{\tau}_{k'}}\fcar_{|\hat{\t}_k|>\hat{\tau}_k} \Big].
\end{align}
Then, by the Cauchy-Schwarz inequality, 
\begin{align}
&\esp_{\bt,\s} \Big[\Big( \sum_{j\in \mathcal S^\complement} a_j(\hat \bt)\fcar_{|\hat{\t}_j|>\hat{\tau}_j}\Big)^2 \Big]\\
&\qquad \le
	\frac{C}{n^2} \esp_{\bt,\s}\Big[(\s^4+\|\bu\|_2^4)\sum_{k\in \mathcal S^\complement} \fcar_{|\hat{\t}_k|>\htau_j}\Big] + \frac{C}{n^2} \esp_{\bt,\s} \Big( \sum_{\substack{k',k \in \mathcal{S}^\complement \\ k'\neq k}} u_{k'}^2 u_k^2\fcar_{|\hat{\t}_{k'}|>\hat{\tau}_{k'}}\fcar_{|\hat{\t}_k|>\hat{\tau}_k} \Big) \\
&\qquad \le \frac{C}{n^2} \Big[\s^4 + \sqrt{\esp_{\bt,\s} \|\bu\|_2^8} \Big] \sum_{k\in \mathcal S^\complement} \sqrt{\prob_{\bt,\s}(|\hat{\t}_k|>\htau_k)}. \label{eq:33a}
\end{align}
Using here Lemmas~\ref{lem:Yuhao} and~\ref{lem:indic} we conclude that
\begin{equation}\label{S1}
\esp_{\bt,\s} \Big[\Big( \sum_{j\in \mathcal S^\complement} a_j(\hat \bt)\fcar_{|\hat{\t}_j|>\hat{\tau}_j}\Big)^2 \Big]
	 \le C\s^4\frac{s^2}{n^2}.
\end{equation}

Next, we consider the second sum on the right-hand side of~\eqref{theo:plen_sparse_1}. Similar to the proof of Theorem~\ref{theo:general} we obtain 
\begin{equation}\label{S2_decomp}
	\esp_{\bt,\s} \Big(\sum_{j\in \mathcal S} a_j(\hat \bt) \fcar_{|\hat{\t}_j|>\hat{\tau}_j}-\t_j^2\Big)^2\le C(S_{1}+S_{2}+S_{3}),
\end{equation}
where 
\begin{equation}
	S_{1}=\esp_{\bt,\s} \Big[\sum_{j\in \mathcal S} \t_j u_j \fcar_{|\hat{\t}_j|>\htau_j} \Big(1-\frac{\|\bX_j\|_2^2}{n} \Big)+ \t_j\frac{\bX_j^T}{n}(\s\bxi-\bw_j)\fcar_{|\hat{\t}_j|>\htau_j}\Big]^2$$
$$ S_{2}=\esp_{\bt,\s} \Big[\frac{2}{n(n-1)}\sum_{i< j} \bar{h} \big((\bX^{i},\xi_{i}),(\bX^{j},\xi_{j})\big) \Big]^2, \quad  S_{3}=\esp_{\bt,\s} \Big[\sum_{j\in \mathcal S} \t_j^2\fcar_{|\hat{\t}_j|\le \htau_j}\Big]^2.
\end{equation}
and $\bar{h}$ is the equivalent of $\tilde{h}$ for indices in $\mathcal S$ instead of $\mathcal S^\complement$. Thus, $S_2$ does not exceed the
expression in \eqref{eq:33a} where $\mathcal S^\complement$ is replaced by $\mathcal S$:
\begin{align}\label{eq:33b}
S_2
& \le \frac{C}{n^2} \Big[\s^4 + \sqrt{\esp_{\bt,\s} \|\bu\|_2^8} \Big] \sum_{k\in \mathcal S} \sqrt{\prob_{\bt,\s}(|\hat{\t}_k|>\htau_k)}. 
\end{align}
  Bounding in this expression $\prob_{\bt,\s}(|\hat{\t}_k|>\hat{\tau}_k)$ by $1$ we get
\begin{equation}\label{S22}
	S_{2}\le C\s^4\frac{s}{n^2}.
\end{equation}
Next, arguing as in~\eqref{varcrossterm1} and~\eqref{varcrossterm2} and using Lemma~\ref{lem:Yuhao} we obtain
\begin{equation}\label{S21}
	S_1 \le \frac{C}{n}\|\bt\|_2^2\esp_{\bt,\s}(\|u\|_2^2+\s^2) \le \frac{C\s^2}{n}\|\bt\|_2^2.
\end{equation}
Finally, introducing the notation $\mathbb B = (\XX_1^T\XX_1)^{-1}$  we find that the term $S_3$ satisfies
\begin{align}\label{S3}
	S_{3} &= \esp_{\bt,\s}  \Big[ \Big(\sum_{j\in \mathcal S} \t_j^2\fcar_{\{|\hat{\t}_j|\le \htau_j, |\t_j|\le 2\htau_j\}}+\sum_{j\in \mathcal S} \t_j^2\fcar_{\{|\hat{\t}_j|\le \htau_j, |\t_j|> 2\htau_j\}}\Big)^2\Big]
	\\ \nonumber
	&\le C \esp_{\bt,\s} \sum_{j,k \in \mathcal S} \big\{ \htau_j^2\htau_k^2 + \t_j^2\t_k^2 \fcar_{\{|\hat{\t}_j|\le \htau_j, |\t_j|> 2\htau_j\}} \fcar_{\{|\hat{\t}_k|\le \htau_k, |\t_k|> 2\htau_k\}} \big\} \phantom{\Big(\sum_{j\in \mathcal S} \t_j^2\sqrt{\prob_{\bt,\s}}} 
	\\ \nonumber
	&\le C   \sum_{j,k\in \mathcal S} \esp_{\bt,\s} \Big(\hat{\s}^4B_{jj}B_{kk}\Big)\log^2\Big(1+\frac{p}{s^2}\Big)+ C\esp_{\bt,\s} \Big[\Big(\sum_{j\in \mathcal S} \t_j^2 \prob_{\bt,\s}^{1/2}(|\hat{\t}_j|\le \htau_j, |\t_j|> 2\htau_j| \XX_1)\Big)^2\Big].
\end{align}
Here, 
$$
\sum_{j,k\in \mathcal S}\esp_{\bt,\s} \Big(\hat{\s}^4B_{jj}B_{kk}\Big) \le \esp_{\bt,\s} (\hat{\s}^4 s^2 \l_{\max}^2(\mathbb B)) \le 
 s^2 (\esp_{\bt,\s} (\hat{\s}^8))^{1/2} (\esp_{\bt,\s} (\l_{\max}^4(\mathbb B)))^{1/2} \le C\s^4\frac{s^2}{n^2},
$$
where the last inequality follows from Lemmas~\ref{lem:Yuhao} and~\ref{lem:hatsigma_ols}. Furthermore, using  Lemmas~\ref{lem:Yuhao} and ~\ref{lem:indic} we get 
\begin{align}
\esp_{\bt,\s} \Big[\Big(\sum_{j\in \mathcal S} \t_j^2\prob_{\bt,\s}^{1/2}(|\hat{\t}_j|\le \htau_j, |\t_j|> 2\htau_j| \XX_1)\Big)^2\Big]
&\le 
\esp\Big[\Big(\s^2 \sum_{j\in \mathcal S} B_{jj} \max_{x>0} \{ x^2\exp(-Cx^2)\} \Big)^2\Big]\\
&\le
C \s^4s^2\esp^2 \big(\l_{\max}^2(\mathbb B)\big)\le C\s^4\frac{s^2}{n^2}.
\end{align}
It follows that
\begin{align}\label{S23}
	S_3
	&\le C\s^4\frac{s^2}{n^2}\log^2\Big(1+\frac{p}{s^2}\Big). \phantom{\sum_{j,k \in \mathcal S}^j\frac{s^4}{•}} 
\end{align}
Combining \eqref{theo:plen_sparse_1},  \eqref{S1} -- \eqref{S23} yields the bound \eqref{eq:Q} of the theorem for $\hat{Q}_S^{LD}$.
The bound \eqref{eq:N} for $\hat{\L}_S^{LD}$  is deduced from \eqref{eq:Q} exactly in the same way  as it is done in the proof of Theorem 8 in~\cite{Collier2017}.

\section{Proofs for Section~\ref{sec:pgen}}

\subsection{Proof of Proposition~\ref{proposition_derumigny}}

We use a combination of arguments from~\cite{BellecLecueTsybakov2017} and~\cite{Derumigny2017}. By Theorem~6.1 in~\cite{Derumigny2017}, if $\bxi$ is standard Gaussian and $\XX$ satisfies weighted restricted eigenvalue condition then the result of Proposition~\ref{proposition_derumigny} holds. Now, note that, by Theorem~8.3 in~\cite{BellecLecueTsybakov2017},  if $\k_0$ in Condition~\ref{6}  is small enough and Condition~\ref{5} holds then $\XX$ satisfies weighted restricted eigenvalue condition  with probability at least $1-3\exp(-Cn)$ for some constant $C>0$ depending only on $M$. 
Thus, Proposition~\ref{proposition_derumigny} holds when $\bxi$ is standard Gaussian. In order to extend it to subGaussian $\bxi$, we note that,  in the proof of Theorem~6.1 in~\cite{Derumigny2017}, the only ingredients involving $\bxi$ are 
Lemmas~7.6 and~7.7 and those lemmas remain valid for subGaussian noise (with possibly different constants depending on $L$). Indeed,  Lemma~7.6 in~\cite{Derumigny2017} states that $\|\bxi\|_2^2/n$ is between two  absolute constants with probability  at least $1-C'\exp(-n/C')$, which remains true for subGaussian  $\bxi$. As concerns Lemma~7.7, it can be replaced by Theorem~9.1 from~\cite{BellecLecueTsybakov2017}  that we state here in the following form. 
\begin{proposition}[Adapted from Theorem~9.1 in~\cite{BellecLecueTsybakov2017}]\label{proposition_tool}
	Let $t>0$ and let $\mathbb{X}$ be a design matrix such that $\max_{1\le i \le p} \|\bX_i\|_2^2\le2n$. Assume that Condition~\ref{1} holds. Then, there is a constant $C>0$ such that, for all $\bu\in\RR^p$,
	\begin{equation}
	\frac1n |\bxi^T \mathbb{X}\bu| \le C \big(G(\bu)\vee \|\bu\|_*\big)
	\end{equation}
	with probability at least $1-e^{-t}$, where $G(\bu) = \frac{\sqrt{t}+1}n \ \|\mathbb{X}\bu\|_2.$
\end{proposition} 
It is not hard to check that the assumption $\max_{1\le i \le p} \|\bX_i\|_2^2\le2n$ of Proposition~\ref{proposition_tool} is satisfied with probability  at least $1-C'\exp(-n/C')$ provided $\bX_i$ are $M$-subGaussian, with components having variance 1,  and Condition~\ref{6} holds.
It follows that there is an event with probability at least $1-C'\exp (-(n\wedge s\log(ep/s))/C' )$ such that the conclusions of Theorem~6.1 in~\cite{Derumigny2017}  (to within the constants that can depend only on $L$, 
$M$ and $\k_0$) hold under the assumptions of Proposition~\ref{proposition_derumigny}.
 
\subsection{Proof of Proposition~\ref{prop_sigma_sqs}}

For brevity,  in this proof we write $(\XX,\bY)$ instead of $(\XX_1,\bY_1)$. Using the definition of $\hat{\s}_{\sf srs}^2$ in~\eqref{definition_sigmaSRS}  we get
\begin{equation}\label{develop}
	|\hat{\s}_{\sf srs}^2-\s^2| \le \frac1{n}\|\mathbb{X}(\hat{\bt}_{\sf srs}-\bt)\|_2^2 + \frac{\s^2}{n} \Big|\|\bxi\|_2^2-n\Big| + \frac{2\s}{n}\Big|\bxi^T\mathbb{X}(\hat{\bt}_{\sf srs}-\bt)\Big|.
\end{equation}
To control the second term on the right-hand side of \eqref{develop}, we apply Bernstein's inequality (\cf e.g. Theorem 2.8.2 in \cite{vershyninbook}): if Condition~\ref{1} holds, then
\begin{equation}\label{chi2}
	\prob\Big(\big|\|\bxi\|_2^2-n\big|>u\Big) \le 2e^{-c_0\big(\frac{u^2}{n}\wedge u\big)}, \quad \forall u>0,
\end{equation}
where $c_0>0$ depends only on $L$.
This yields, for any $t>0$ such that $t \le c_0n$,
\begin{equation}\label{eqq1}
	\big|\|\bxi\|_2^2-n\big| \le \sqrt{nt/c_0}
\end{equation}
with probability at least $1-2e^{-t}$. Next, to bound the first and the third terms on the right-hand side of~\eqref{develop} we place ourselves on the event of probability at least $1-C'\exp (-(n\wedge s\log(ep/s))/C' )$ where the result of  Proposition~\ref{proposition_derumigny} holds. We denote this event by ${\cal C}$. By Proposition~\ref{proposition_derumigny}, on ${\cal C}$ we have 
$\frac1{n}\|\mathbb{X}(\hat{\bt}_{\sf srs}-\bt)\|_2^2\le C' \s^2\frac{s \log(ep/s)}{n}$. 

Finally, to bound the third term on the right-hand side of \eqref{develop} we use Proposition~\ref{proposition_tool} where we set $\bu=\hat{\bt}_{\sf srs}-\bt$. For $t>0$, denote by ${\cal C}'_t$ the intersection of ${\cal C}$ with the event of probability at least $1-e^{-t}$, on which Proposition~\ref{proposition_tool} holds. Then $\prob({\cal C}'_t) \ge 1-C'\exp (-(n\wedge t \wedge s\log(ep/s))/C' )$.  Proposition~\ref{proposition_derumigny} yields that, for any $t>0$ satisfying $t\le c_0n$ we have, on the event~${\cal C}$,
\begin{align}
	\|\bu\|_* &\le C \s\frac{s\log(ep/s)}{n},\\
	G(\bu) & \le \frac{\s t}n + \frac{\|\mathbb{X}\bu\|_2^2}{\s n} + \frac{\|\mathbb{X}\bu\|_2}{ n} \le C
	\s\Big( \frac{t}n + \frac{s \log(ep/s)}{n}+\frac{(s \log(ep/s))^{1/2}}{n} \Big)\\
	& \le  C' 
	\s\Big( \sqrt{\frac{t}n} + \frac{s \log(ep/s)}{n}\Big),
\end{align}
 where we have used the fact that $s \log(ep/s)\ge  1$.
These remarks and Proposition~\ref{proposition_tool} imply that,  for any $t>0$ satisfying $t\le c_0n$ we have, on the event~${\cal C}'_t$,
\begin{equation}\label{eqq3}
	\frac1n |\bxi^T \mathbb{X}\bu| \le C \s \Big(\frac{t}n + C' \frac{ s \log(ep/s)}{n}\Big).
\end{equation}
The result of the proposition follows.
 
\subsection{Proof of Theorem~\ref{theo:pgen:dense}} 


For brevity,  in this subsection we write $(\XX,\bY)$, $\hat Q$ and $\hat \L$ instead of $(\XX_2,\bY_2)$, $\hat Q_D^{HD}$ and $\hat \L_D^{HD}$, respectively. We also set $\bu = \hat \bt_{\sf srs} - \bt$ and $\bw_j=\sum_{k\neq j} u_k\bX_{k}$. 

Recalling the expressions obtained in the proof of Theorem~\ref{theo:general}, we have $\hat{Q} = f_1(\bZ) + f_2(\bZ)$, where $\bZ=\{ (X_{ij})_{1\le i\le n, 1\le j\le p}, (\xi_i)_{1\le i\le n}\}$ and
\begin{align}
	f_1(\bZ) &= Q(\bt) + \sum_{j=1}^p \Big[ 2\t_ju_j\Big(1-\frac{\|\bX_{j}\|_2^2}{n} \Big)+2\t_j \frac{\bX_{j}^T}{n}(\s\bxi-\bw_j)\Big] + \frac{2}{n} \sum_{j=1}^p  u_j\sum_{i=1}^n X_{ij}(\s \xi_i - \bX^i\bu), \\
	f_2(\bZ) &= \sum_{i=1}^p u_i^2 + \frac1{n(n-1)} \sum_{i\neq j} \sum_{k=1}^p X_{ik}X_{jk}(\s \xi_{i}-\mathbf{X}^i\bu)(\s \xi_j-\mathbf{X}^j\bu).
\end{align}
We will use a result from~\cite{2015_adamszak}. To this end, we introduce the following notation. Denote by $\mathcal P_d$ the set of all partitions of $\{1,\ldots,d\}$  into nonempty, pairwise disjoint sets. For a partition $\mathcal I = \{\mathcal I_1,\dots, \mathcal I_K\}\in \mathcal P_d$ and
a multi-indexed array $\mathfrak{M} = (M_{i_1,\ldots,i_d})_{i_1,\ldots,i_d=1,\ldots,q}$, we define 
\begin{equation}
	\|\mathfrak{M}\|_{\mathcal{I}} := \max\Big\{ \sum_{i_1,\ldots,i_d=1}^q M_{i_1,\ldots,i_d} \prod_{k=1}^K x_{{\bf i}_k} : \, \|x_{{\bf i}_k}\|_2  \le 1, \, k=1,\dots, K \Big\},
\end{equation}
where $x_{{\bf i}_k}$ 
is a $\card(\mathcal{I}_k)$-dimensional array indexed by ${\bf i}_k=(i_j, j\in \mathcal{I}_k)$. Here,  $\card(\mathcal{I}_k)$ is the cardinality of $\mathcal{I}_k$ and $\|x_{{\bf i}_k}\|_2^2= \sum_{|{\bf i}_k|\le q} x_{{\bf i}_k}^2$ with $|{\bf i}_k| = \max(i_j, j\in \mathcal{I}_k)$. 
 Examples of this notion can be found in~\cite{2015_adamszak}. 
Finally, denote by ${\mathbb D}^d f$ the $d$-th derivative of a function $f:\RR^q\to \RR$, which we identify with a multi-indexed array $\mathfrak{M}$, where only entries with indices satisfying $i_1+\cdots + i_d= d$ (that are the corresponding partial derivatives) can be non-zero, and all other entries vanish.

\begin{proposition}[Theorem 1.4 in \cite{2015_adamszak}]\label{Adamczak-Wolff}
	Let $f:\RR^q\to \RR$ be a polynomial of $q$ variables of degree $D$. Let $\bZ=(Z_1,\ldots, Z_q)$ be a vector with independent $L$-subGaussian components. Then, for any $t>0$, 
	\begin{equation}
		\prob\big(|f(\bZ)-\esp f(\bZ)|\ge t\big)\le 2\exp\Big[-c_D \min_{1\le d\le D} \min_{{\cal I}\in {\mathcal P}_d}\Big(\frac{t}{L^d\|\esp {\mathbb D}^d f(\bZ)\|_{\cal I}}\Big)^{\frac{2}{{\rm \bf Card}(\cal I)}}\Big],
	\end{equation} 
	where $c_D$ is a positive constant depending only on $D$. 
\end{proposition}
  
We now apply this proposition with $f=\hat{Q}$, which is a polynomial of degree $D=4$ (assuming $\bu$ fixed) of $q=np + n$ variables. Note that $\esp_{\bt,\s} ({\mathbb D}\hat{Q}\suchthat \bu) = \esp_{\bt,\s} ({\mathbb D}^3\hat{Q}\suchthat \bu) = 0$, and $\esp_{\bt,\s} ({\mathbb D}^4f_1(\bZ)\suchthat \bu) = 0$. Hence, in order to apply Proposition \ref{Adamczak-Wolff}, we only need to evaluate $\|\esp_{\bt,\s} ({\mathbb D}^2f_i(\bZ)\suchthat \bu)\|_{\mathcal I}$, $i=1,2$,  for every $\mathcal I \in \mathcal P_2$ and $\|\esp_{\bt,\s} ({\mathbb D}^4f_2(\bZ)\suchthat \bu)\|_{\mathcal I}$ for every $\mathcal I \in \mathcal P_4$.
  
We have
\begin{align}\label{eq:diff2}
	&\frac{\partial^2 f_1}{\partial \xi_i \partial \xi_j}(\bZ) = 0, \quad \frac{\partial^2 f_1}{\partial \xi_i \partial X_{jk}}(\bZ) = \frac{2\s}{n} (u_k+\t_k) \fcar_{i=j}, \\
	&\frac{\partial^2 f_1}{\partial X_{ij} \partial X_{kl}}(\bZ) =  -\frac{2}{n} \big[ u_l(u_j+\t_j) + u_j(u_l+\t_l)\big] \fcar_{i=k}.
\end{align}
Since $\|\cdot\|_{\{1, 2\}}$ is simply the Euclidean norm of the array (\cf\cite{2015_adamszak}) we find
\begin{equation}\label{theo:pgen:dense1}
	\|\esp_{\bt,\s} \big({\mathbb D}^2f_1(\bZ)|\bu\big)\|^2_{\{1, 2\}}\le C \frac{\big(\|\bu\|_2^2+\s^2\big)\big(\|\bu\|_2^2+\|\bt\|_2^2\big)}{n}.
\end{equation}
Furthermore, 
\begin{align}
	\|\esp_{\bt,\s} \big({\mathbb D}^2f_1(\bZ)&|\bu\big)\|_{\{\{1\}, \{2\}\}} \\ 
	&= \max \Big\{ \frac{2\s}{n} \sum_{i=1}^n \sum_{k=1}^p (u_k+\t_k) x_{i,k} \bar y_{i} - \frac{2}{n} \sum_{i=1}^n \sum_{j,l=1}^p \big[ u_l(u_j+\t_j) + u_j(u_l+\t_l)\big] x_{i,j} y_{i,l} \\
	&+ \frac{2\s}{n} \sum_{i=1}^n \sum_{k=1}^p (u_k+\t_k) y_{i,k} \bar x_{i} - \frac{2}{n} \sum_{i=1}^n \sum_{j,l=1}^p \big[ u_l(u_j+\t_j) + u_j(u_l+\t_l)\big] y_{i,j} x_{i,l}\Big\},
\end{align}
where the maximum is taken over $(x_{i,j}, \bar x_k, y_{i,j},   \bar y_{k})$ such that $\sum_{i=1}^n \sum_{j=1}^p x^2_{i,j} + \sum_{k=1}^n \bar x^2_{k} \le 1$ and $\sum_{i=1}^n \sum_{j=1}^p y^2_{i,j} + \sum_{k=1}^n \bar y^2_{k} \le 1$. By the Cauchy-Schwarz inequality,
\begin{equation}
	\Big| \frac{2\s}{n} \sum_{i=1}^n \sum_{k=1}^p (u_k+\t_k) x_{i,k} \bar y_{i} \Big| \le \frac{2\s}{n} \big(\|\bu\|_2+\|\bt\|_2\big), 
\end{equation}
and on the other hand, 
\begin{align}
	\Big| \frac{2}{n} \sum_{i=1}^n \sum_{j,l=1}^p \big[ u_l(u_j+\t_j) + u_j(u_l+\t_l)\big] x_{i,j} y_{i,l} \Big| &\le \frac4{n} \sqrt{\sum_{i,j,l} u_l^2 x_{i,j}^2} \sqrt{\sum_{i,j,l} (u_j + \t_j)^2 y_{i,l}^2} \\
	&\le \frac{C}{n} \|\bu\|_2 \big( \|\bu\|_2 + \|\bt\|_2 \big).
\end{align}
Hence $\|\esp_{\bt,\s} \big( {\mathbb D}^2 f_1(\bZ) \suchthat \bu\big) \|_{\{\{1\},\{2\}\}} \le Cn^{-1} \big(\|\bu\|_2+\s\big)\big(\|\bu\|_2+\|\bt\|_2\big).$ 

We now turn to 
evaluation of $\|\esp_{\bt,\s} ({\mathbb D}^2f_2(\bZ)\suchthat \bu)\|_{\mathcal I}$ for $\mathcal I \in \mathcal P_2$ and $\|\esp_{\bt,\s} ({\mathbb D}^4f_2(\bZ)\suchthat \bu)\|_{\mathcal I}$ for $\mathcal I \in \mathcal P_4$. We have
\begin{align}\label{dec:f2} 
	f_2(\bZ) - \sum_{i=1}^p u_i^2 &= S_1(\bZ) + S_2(\bZ)+ S_3(\bZ)   
	\end{align}
	where
	\begin{align}
	S_1(\bZ)&=\frac{\s^2}{n(n-1)} \sum_{i,j} \sum_{k_1,l_1} \sum_{k_2,l_2} \big[\fcar_{i \neq j, k_1 = i, k_2 = j, l_1 = l_2}\big] \xi_i\xi_j X_{k_1 l_1} X_{k_2 l_2} \\
	 S_2(\bZ)&=- \frac{2\s}{n(n-1)} \sum_i \sum_{k_1,l_1} \sum_{k_2,l_2} \sum_{k_3,l_3} \big[u_{l_3} \fcar_{k_1 = i, k_2 = k_3, k_2 \neq i, l_1 = l_2}\big] \xi_i X_{k_1 l_1} X_{k_2 l_2} X_{k_3 l_3} \\
	S_3(\bZ) &= \frac{1}{n(n-1)} \sum_{k_1,l_1} \sum_{k_2,l_2} \sum_{k_3,l_3} \sum_{k_4,l_4} \big[u_{l_3} u_{l_4} \fcar_{k_1 = k_3, k_2 = k_4, k_1 \neq k_2, l_1 = l_2} \big] X_{k_1 l_1} X_{k_2 l_2} X_{k_3 l_3} X_{k_4 l_4}.
\end{align}
It follows that
\begin{align}
	&\esp_{\bt,\s} \Big[ \frac{\partial^2 f_2}{\partial \xi_i \partial \xi_j}(\bZ) \suchthat \bu \Big] = 0, \quad \esp_{\bt,\s} \Big[ \frac{\partial^2 f_2}{\partial \xi_i \partial X_{k,l}}(\bZ) \suchthat \bu \Big] = -\frac{2\s u_l}{n} \fcar_{i = k}, \\
	&\esp_{\bt,\s} \Big[ \frac{\partial^2 f_2}{\partial X_{k_1,l_1}\partial X_{k_2,l_2}}(\bZ) \suchthat \bu \Big] = \frac{2u_{l_1}u_{l_2}}{n} \fcar_{k_1 = k_2}. 
\end{align}
By replacing the $u_j + \t_j$ in~\eqref{eq:diff2} with $u_k$, similar argument as for $f_1$ yields $\|\esp_{\bt,\s} \big({\mathbb D}^2f_2(\bZ)|\bu\big)\|_{\{1, 2\}}^2\le Cn^{-1} \big(\|\bu\|_2^2+\s^2\big)\|\bu\|_2^2$ and $\|\esp_{\bt,\s} \big( {\mathbb D}^2 f_2(\bZ) \suchthat \bu\big) \|_{\{\{1\},\{2\}\}} \le Cn^{-1} \big(\|\bu\|_2+\s\big)\|\bu\|_2.$

Furthermore, retrieving the fourth order derivatives of $f_2$ from~\eqref{dec:f2} and recalling that $\|\cdot\|_{\{1,2,3,4\}}$ is the Euclidean norm we find
\begin{equation}
		\|\esp_{\bt,\s} \big( {\mathbb D}^4\hat{Q} \suchthat \bu\big) \|^2_{\{1,2,3,4\}} = \|\esp_{\bt,\s} \big( {\mathbb D}^4f_2(\bZ) \suchthat \bu\big) \|^2_{\{1,2,3,4\}} \le \big(\|\bu\|_2^4+\s^4 \big)\frac{p}{n^2}.
\end{equation}
Next,  for any partition $\mathcal{I}\in\mathcal{P}_4$,
\begin{equation}\label{dec:norm:f2}
	\|\esp_{\bt,\s} \big( {\mathbb D}^4 f_2(\bZ) \suchthat \bu\big) \|_{\mathcal{I}} \le \sum_{i=1}^3 \|\esp_{\bt,\s} \big( {\mathbb D}^4 S_i(\bZ) \suchthat \bu\big) \|_{\mathcal{I}}.
\end{equation}
We now evaluate separately the three terms on the right hand side of this inequality. We will only do it for the partition 
$\mathcal{I}=\{\{1\},\{2,3,4\}\}$ since other partitions are treated analogously. First, consider the term with $S_1(\bZ)$ in \eqref{dec:norm:f2}. 
Observe that
\begin{align}\label{eq:diff3}
\|\esp_{\bt,\s} \big( {\mathbb D}^4 S_1(\bZ) \suchthat \bu\big) \|_{\{\{1\},\{2,3,4\}\}} \leq  \frac{C\s^2}{n(n-1)} \max \Big\{\sum_{i\neq j} \sum_{l = 1}^p  x_i y_{i,j,l} + \sum_{i\neq j} \sum_{l = 1}^p x_{i,l} y_{i,j,l}\Big\},
\end{align}
where the maximum is taken under the constraints imposed by the definition of the norm $ \|\cdot\|_{\mathcal{I}}$. Thus, for the first double sum on the right hand side of~\eqref{eq:diff3}, the maximum is taken under the constraint $\sum_{i=1}^n x_i^2 \le 1, \sum_{i\neq j} \sum_{l=1}^p y_{i,j,l}^2 \le 1$. Using the implication
\begin{align}
	\sum_{i=1}^n x_i^2 \le 1, \sum_{i\neq j} \sum_{l=1}^p y_{i,j,l}^2 \le 1 \quad\donc\quad \sum_{i\neq j} \sum_{l=1}^p x_i y_{i,j,l} \le \sqrt{np}
\end{align}
to control the first double sum and treating quite analogously the second double sum in \eqref{eq:diff3}  we find that 
$\|\esp_{\bt,\s} \big( {\mathbb D}^4 S_1(\bZ) \suchthat \bu\big) \|_{\{\{1\},\{2,3,4\}\}} \le C\s^2\sqrt{p}n^{-3/2}$. 

 Next, we consider the term with $S_2(\bZ)$ in \eqref{dec:norm:f2}. Similarly to~\eqref{eq:diff3} we have
 \begin{align}\label{eq:diff4}
 \|\esp_{\bt,\s} & \big( {\mathbb D}^4 S_2(\bZ) \suchthat \bu\big) \|_{\{\{1\},\{2,3,4\}\}} \le \frac{C\s}{n(n-1)} \max \Big\{\sum_{i\neq j} \sum_{l, l' = 1}^p x_i y_{i,j,l,l'} u_{l'} \\
 & +  \sum_{i\neq j} \sum_{l, l' = 1}^p x_{i,l} y_{i,j,l,l'} u_{l'}  +  \sum_{i\neq j} \sum_{l, l' = 1}^p x_{i,l'} y_{i,j,l} u_{l'} \Big\},
 \end{align} 
 where the maximum is taken under the constraints imposed by the definition of the norm $ \|\cdot\|_{\mathcal{I}}$. For the constraint corresponding to the first triple sum on the  right hand side of~\eqref{eq:diff4} we have the implication
\begin{align}
	\sum_{i=1}^n x_i^2 \le 1, \sum_{i\neq j} \sum_{l,l'=1}^p y_{i,j,l,l'}^2 \le 1 \quad\donc\quad \sum_{i\neq j} \sum_{l=1}^p x_i y_{i,j,l,l'} u_{l'} \le \sqrt{\sum_{i,j,l,l'} u_{l'}^2x_i^2 \sum_{i,j,l,l'} y_{i,j,l,l'}^2} \le \|\bu\|_2 \sqrt{np}.
\end{align}
The other terms on the right hand side of~\eqref{eq:diff4} are treated analogously. It follows that 
\[
\|\esp_{\bt,\s} \big( {\mathbb D}^4 S_2(\bZ) \suchthat \bu\big) \|_{\{\{1\},\{2,3,4\}\}} \le C\s\|\bu\|_2\sqrt{p}n^{-3/2}.
\]
Similarly, such a bound holds for every partition in $\mathcal{P}_4$. Finally,  the control of the term with $S_3(\bZ)$ in \eqref{dec:norm:f2} follows the same lines by using the implication
\begin{align}
	\sum_{i,a} x_{i,a}^2 \le 1, \sum_{i\neq j} \sum_{a,b,c} y_{i,j,a,b,c}^2 \le 1 \quad\donc\quad \sum_{i,j,a,b,c} x_{i,a} y_{i,j,a,b,c} u_{b}u_{c} &\le \sqrt{\sum_{i,j,a,b,c} u_{b}^2x_{i,a}^2 \sum_{i,j,a,b,c} u^2_c y_{i,j,a,b,c}^2} \\
	&\le \|\bu\|_2^2\sqrt{np},
\end{align}
and analogous implications to obtain the bound $\|\esp_{\bt,\s} \big( {\mathbb D}^4 S_3(\bZ) \suchthat \bu\big) \|_{\{\{1\},\{2,3,4\}\}} \le C\|\bu\|_2^2\sqrt{p}n^{-3/2}$. Again, such a bound holds for every partition in $\mathcal{P}_4$ (we skip the argument, which is quite analogous). 

Using the bounds obtained above for the three terms in~\eqref{dec:norm:f2} we get
\begin{equation}
	\max_{\mathcal{I} \in \mathcal{P}_4} \|\esp_{\bt,\s} \big( {\mathbb D}^4 \hat{Q} \suchthat \bu\big) \|_{\mathcal{I}} \le C\big(\|\bu\|_2^2+\s^2\big)\sqrt{p}n^{-3/2}.
\end{equation}
Using Proposition~\ref{Adamczak-Wolff} and taking $t=\frac{\s^2\sqrt{pv}}{n}+\s\sqrt{\frac{v}{n}}\|\bt\|_2$ with $0<v\le n^{1/3}$ we have on the event $\big\{ \|\bu\|_2^2 \le c_{\sf srs, 2} \sigma^2 s\log(ep/s)/n \big\}$ and under Condition~\ref{6} that
\begin{align}
	\min_{1\le d\le D} \min_{{\cal I}\in {\cal P}_d}\Big(\frac{t}{L^d\|\esp {\mathbb D}^d f(\bZ)\|_{\cal I}}\Big)^{\frac{2}{{\rm \bf Card}(\cal I)}} &\ge C \Big[\frac{nt^2}{\big(\|\bu\|_2^2+\s^2\big)\big(\|\bu\|_2^2+\|\bt\|_2^2\big)} \wedge \frac{nt}{\big(\|\bu\|_2+\|\bt\|_2\big)\big(\s+\|\bu\|_2\big)} \\ 
	&\wedge \frac{n^2t^2}{\big(\|\bu\|_2^4+\s^4\big) p} \wedge \min_{k=2,3,4}\Big\{\frac{n^{3/2}t}{\big(\|\bu\|_2^2+\s^2\big)\sqrt{p}}\Big\}^{2/k} \Big] \\
	&\ge Cv.   \phantom{\frac{n^2t^2}{\big(\|\bu\|_2^4+\s^4\big) p}}
\end{align}
Thus, the first assertion of the theorem follows from Propositions~\ref{proposition_derumigny} and~\ref{Adamczak-Wolff}, and the fact that the estimator $\hat{Q}$ is conditionally unbiased: $\esp_{\bt,\s} \big(\hat{Q} \suchthat \bu\big) = Q(\bt)$. 

{We now prove the second assertion in the theorem. Set $\tau^2 = \sigma^2 \frac{\sqrt{p v}}{n}$. We bound $|\hat{\L} - \|\bt\|_2|$ by considering the cases $\| \bt \|_2 \leq \tau$ and $\| \bt \|_2 > \tau$ separately.}

{
Case I: $\| \bt \|_2 \leq \tau$. Then we have
\begin{equation}
|\hat{\L} - \|\bt\|_2|^2 \leq 2 (\hat{\L}^2 + \|\bt\|_2^2) = 2 (\hat{Q} - \|\bt\|_2^2) + 4 \|\bt\|_2^2.
\end{equation}
By combining this inequality with the first assertion of the theorem we get that, with probability at least $1-C'\big[e^{-v}+e^{-\frac{n\wedge s\log(ep/s)}{C'}}\big]$, 
\begin{align}\label{eq:ntau1}
|\hat{\L} - \|\bt\|_2|^2 & \leq C \Big(\sqrt{\s^4 \frac{pv}{n^2} + \s^2 \frac{\tau^2 v}{n}} + \tau^2\Big) \\
& \leq C(\s^2 \frac{\sqrt{pv}}{n} + \s \frac{\tau \sqrt{v}}{\sqrt{n}} + \tau^2) \leq C(\s^2 \frac{\sqrt{pv}}{n} + \s^2 \frac{p^{1/4} v^{3/4}}{n}).
\end{align}

{Case II: $\| \bt \|_2 > \tau$. Applying the inequality $\forall a >0, b \geq 0, (a - b)^2 \leq (a^2 - b^2)^2 / a^2$ we find
\begin{align}
|\hat{\L} - \|\bt\|_2|^2 & \leq \frac{(\hat{Q} - \|\bt\|_2^2)^2}{\|\bt\|_2^2}.
\end{align}
Combining this inequality with the first assertion of the theorem we obtain  that, with probability at least $1-C'\big[e^{-v}+e^{-\frac{n\wedge s\log(ep/s)}{C'}}\big]$, 
\begin{align}\label{eq:ntau2}
|\hat{\L} - \|\bt\|_2|^2 & \leq C \frac{\s^4 \frac{pv}{n^2} + \s^2 \frac{\|\bt\|_2^2 v}{n}}{\|\bt\|_2^2} \leq C (\s^4 \frac{pv}{\tau^2 n^2} + \s^2 \frac{v}{n}) \leq C (\s^2 \frac{\sqrt{pv}}{n} + \s^2 \frac{v}{n}).
\end{align}
Putting~\eqref{eq:ntau1} and~\eqref{eq:ntau2} together we find
\begin{equation}
	\prob_{\bt,\s}\Big(|\hat{\L}-\|\bt\|_2|>C\frac{\s(pv)^{1/4}}{\sqrt{n}}+C\s\sqrt{\frac{v}{n}} + C \s \frac{v^{3/8}p^{1/8}}{\sqrt{n}}\Big)\le C'(e^{-v/C'} + (\frac{s}{ep})^{s/C'}).
\end{equation}
Finally, since we assume that $v \leq n^{1/3}$ and $p \geq \gamma n$ for some $0 <\gamma < 1$, the term $\frac{\s(pv)^{1/4}}{\sqrt{n}}$ in the above inequality dominates the other two terms. The second assertion of the theorem follows.


\subsection{Proof of Theorem~\ref{theo:pgen:sparse12}}

For brevity,  in this subsection we write $\hat \bt, \tilde \bt, \hat{\s}$ and $\hat Q$ instead of $\hat{\bt}_{\sf srs}, \tilde{\bt}_{\sf srs}, \sqrt2 \hat{\s}_{\sf srs}$ and $\hat Q_S^{HD}$,  respectively. We also set $\bu = \hat \bt - \bt$ and
\begin{equation}\label{def_lem4}
	\bnu = \Big(\mathbb I_p - \frac{\XX_3^T\XX_3}{n}\Big) \bu, \ \beps = \frac{\s}{n}\XX_3^T \bxi_3, \ \hat{\tau}=\alpha\hat{\s} \sqrt{ \frac{\log(1+p/s^2)}{n}}, \ \tau=\alpha\s \sqrt{\frac{ \log(1+p/s^2)}{n}}.
\end{equation}
With this notation, we have $\tilde{\bt}=\bt+\bnu+\beps$. We define the random event
\begin{equation}\label{def_calA}
	\calA = \Big\{ \|\hat{\bt}-\bt\|_2^2 \le \s^2\Big\}\cap
		\Big\{\s^2\le \hat{\s}^2\le 3 \s^2\Big\}.
\end{equation}

\subsubsection{Preliminary lemmas for the proof of Theorem~\ref{theo:pgen:sparse12}}
          
\begin{lemma}\label{eps}
	Let Conditions~\ref{1},~\ref{5} and~\ref{6}  hold. Then  there exist positive constants $C_1,C_2$ depending only on $L,M$ such that for $\a>0$ large enough and any $j\in \{1,\ldots, p\}$,
	\begin{equation}
		\prob_{\bt,\s} \Big(\big\{|\e_j|>\hat{\tau}/2 \big\} \cap \calA \Big) \le C_1 \frac{s^2}{p}
	\end{equation}
     and
     \begin{equation}
     	\t_j^4\prob_{\bt,\s} \big(|\e_j| > |\t_j|/4\big) \le C_2 \frac{\s^4}{n^2},
     \end{equation}     
     where $\beps, \htau$ and $\calA$ are defined in~\eqref{def_lem4} and~\eqref{def_calA} respectively.
\end{lemma} 

\begin{proof}
	For brevity, in this proof we denote by $\bX_j$ the $j$th column of matrix $\XX_3$. 
	Using the definition of $\calA$ we get
	\begin{align}
		\prob_{\bt,\s} \Big(\big\{|\e_j|>\hat{\tau}/2 \big\} \cap \calA \Big) &\le \prob_{\bt,\s}\big(|\e_j|>\tau/2\big) \\
		&\le \prob_{\bt,\s} \big( \|\bX_j\|_2^2 > 2n\big) + \prob_{\bt,\s} \big( \|\bX_j\|_2^2 \le 2n, |\e_j|>\tau/2\big).
	\end{align}
	Since $\xi_1,\ldots, \xi_n$ are independent $L$-subGaussian variables we have that $\e_j$ is a $\s L\|\bX_j\|_2/n$-subGaussian variable for fixed $\bX_j$.  It follows that 
	\begin{equation}
		\prob_{\bt,\s}\big(|\e_j|>\tau/2 \suchthat \bX_j\big)\le 2 \exp\Big(-\frac{n^2\tau^2}{8L^2\s^2\|\bX_j\|_2^2}\Big),
	\end{equation}
	and thus
	\begin{equation}
	\prob_{\bt,\s} \big( \|\bX_j\|_2^2 \le 2n, |\e_j|>\tau/2\big)\le 2 \exp\Big(-\frac{n\tau^2}{16L^2\s^2}\Big)
\le 2s^2/p
	\end{equation}
	if $\alpha>4L$.
	Moreover,  as $X_{1j}, \ldots, X_{nj}$ are independent $M$-subGaussian variables their squares are subExponential. Since also $\esp(X_{ij}^2)=1$  Bernstein's inequality (\cf Theorem 2.8.2 in \cite{vershyninbook}) yields that there is a constant $C>0$ such that $\prob_{\bt,\s}(\| \bX_j\|_2^2-n> n)\le \exp(-Cn).$ In view of Condition~\ref{6} we have $\exp(-Cn)\le C'\frac{s^2}{p}$ for $\k_0>0$ small enough. The first result of the lemma follows. 
	
	To prove the second inequality of the lemma, note that the variables $(X_{ij}\xi_i)_{1\le i\le n}$ are independent, and they are subExponential as products of two subGaussians, \cf Lemma~2.7.7 in~\cite{vershyninbook}. Thus, Theorem~2.8.2 in~\cite{vershyninbook} yields
	\begin{align}\label{subbexp}
		\t_j^4\prob_{\bt,\s} \big(|\e_j| > |\t_j|/4\big) &\le 2\t_j^4 e^{-C\min \Big\{ \frac{n\t_j^2}{\s^2}, \frac{n|\t_j|}{\s} \Big\}} \le  \frac{2\s^4}{n^2} \sup_{x\ge0} \big[ x^4 e^{-C\min\{x^2, x\}} \big].
	\end{align}
\end{proof}
 
\begin{lemma}\label{nu}
	Let Conditions~\ref{1},~\ref{5} and~\ref{6}  hold. Then  there exist positive constants $C_1,C_2$ depending only on $L,M$ such that for $\a>0$ large enough and any $j\in \{1,\ldots, p\}$,
	\begin{equation}
		\prob_{\bt,\s} \Big(\big\{|\nu_j|>\hat{\tau}/2 \big\} \cap \calA \Big) \le C_1 \frac{s^2}{p}
	\end{equation}
	and
	\begin{equation}
     	\t_j^4\prob_{\bt,\s} \big(\big\{|\nu_j| > |\t_j|/4\big\} \cap \calA \big) \le C_2 \frac{\s^4}{n^2},
     \end{equation}     
     where $\bnu, \htau$ and $\calA$ are defined in~\eqref{def_lem4} and~\eqref{def_calA} respectively.
\end{lemma}

\begin{proof}
	For brevity, in this proof we denote by $\bX_j$ the $j$th column of matrix $\XX_3$,  
	  by $\mathbb{X}_{-j}$ the matrix obtained by removing the $j$th column (that is, $\bX_j$) from $\XX_3$, and by $\bu_{-j}$ the vector obtained by removing the $j$th entry from $\bu$.

	For any $j\in\{1,\ldots,p\}$, we have the upper bound
	\begin{equation}\label{decnu}
	 |\nu_j| \le \Big| \Big(1 - \frac{\|\bX_j\|_2^2}{n}\Big) u_j \Big| + \Big| \frac{1}{n} \bX_j^T\XX_{-j} \bu_{-j} \Big|.
	\end{equation}
	Using the definition of $\calA$ and Bernstein's inequality (Theorem~2.8.2 in~\cite{vershyninbook}) we find
	\begin{align}\label{eq:vspom0}
		\prob_{\bt,\s}\Big(\big\{\big| (1-n^{-1}\|\bX_j\|^2_2)u_j| > \hat{\tau}/4 \big\} \cap \calA\Big) &\le \prob_{\bt,\s}\Big( \big|1-n^{-1}\|\bX_j\|^2_2\big| > \tau/(4\s) \Big) \\
		&\le  2 e^{-Cn\min \{ {\tau^2}/{\s^2}, {\tau}/{\s} \}} \le C's^2/p, \nonumber
	\end{align}
	where the last inequality holds for $\a$ large enough due to the fact that $\tau/\s<1$ for $\kappa_0$ small enough, \cf Condition~\ref{6}. 
	
	Next, notice that $\bX_j^T\XX_{-j} \bu_{-j}=\sum_{i=1}^n \eta_i$ where the random variables $\eta_i:= X_{ij}\sum_{k\neq j} X_{ik} u_k $, $i=1,\dots,n$,  are independent and zero-mean conditionally on $\bu$. Furthermore, conditionally on $\bu$, each variable $\sum_{k\neq j} X_{ik}u_k$ is $\|\bu\|_2M$-subGaussian. Thus, the subGaussian norms (\cf \cite{vershyninbook}) of $X_{ij}$ and of 
	$\sum_{k\neq j} X_{ik}u_k$, conditionally on $\bu$, do not exceed $C_0M$ and $C_0M \|\bu\|_2$, respectively, where $C_0>0$ is an absolute constant. 
	By Lemma~2.7.7 in~\cite{vershyninbook}, the subExponential norm of $\eta_i$ (for any fixed $\bu$) does not exceed $C_0^2M^2\|\bu\|_2$. Therefore, conditionally on $\bu$, we can apply Theorem~2.8.1 in~\cite{vershyninbook}. This yields
		\begin{align}\label{eq:vspom}
	\prob_{\bt,\s}\Big( \Big| \frac{1}{n} \bX_j^T\XX_{-j} \bu_{-j} \Big|> \tau/4 \suchthat \bu \Big) &\le  2 e^{-Cn \min \{ {\tau^2}/{\|\bu\|_2^2}, {\tau}/{\|\bu\|_2} \}} .
	\end{align}
	It follows that
	\begin{align}\label{eq:vspom1}
		\prob_{\bt,\s}\Big( \Big\{ \Big| \frac{1}{n} \bX_j^T\XX_{-j} \bu_{-j} \Big|> \hat{\tau}/4 \Big\} \cap \calA \Big) &\le \prob_{\bt,\s}\Big(  \Big\{ \Big| \frac{1}{n} \bX_j^T\XX_{-j} \bu_{-j} \Big|> \tau/4 \Big\}   \cap \calA \Big)
		\\ & \nonumber
		\le \esp_{\bt,\s}\Big[\prob_{\bt,\s}\Big(  \Big| \frac{1}{n} \bX_j^T\XX_{-j} \bu_{-j} \Big|> \tau/4 \suchthat \bu  \Big) \fcar_{\| \bu\|_2\le \s} \Big]
		 \\ & \le 2 e^{-Cn\min \{ {\tau^2}/{\s^2}, {\tau}/{\s} \}} \le C's^2/p.\nonumber
	\end{align}
	The first assertion of the lemma follows.
	
	To prove the second assertion, we use again the bound~\eqref{decnu}.
Considering the first term on the right hand side of~\eqref{decnu} and reasoning analogously to  \eqref{eq:vspom0} 
we get
	\begin{align}\label{eq:1term}
		\t_j^4 \prob_{\bt,\s}\Big(\big\{\big|(1-n^{-1}\|\bX_j\|^2_2)u_j| > |\t_j|/8 \big\} \cap \calA\Big) \le 2\t_j^4e^{-C\min \Big\{ \frac{n\t_j^2}{\s^2}, \frac{n|\t_j|}{\s} \Big\}} \le C' \frac{\s^4}{n^2},
	\end{align}
	where the last inequality is obtained as in \eqref{subbexp}.
To handle the second term on the right hand side of~\eqref{decnu}, we note that, analogously to \eqref{eq:vspom} and \eqref{eq:vspom1}, 
	\begin{align}
			\t_j^4\prob_{\bt,\s}\Big( \Big| \frac{1}{n} \bX_j^T\XX_{-j} \bu_{-j} \Big|>  |\t_j|/8 \suchthat \bu \Big) &\le  2 \t_j^4 e^{-Cn \min \{ {\t_j^2}/{\|\bu\|_2^2}, {|\t_j|}/{\|\bu\|_2} \}},
	\end{align}
	and 
	\begin{align}\label{eq:vspom2}
	\t_j^4 \prob_{\bt,\s}\Big( \Big\{ \Big| \frac{1}{n} \bX_j^T\XX_{-j} \bu_{-j} \Big|> |\t_j|/8 \Big\} \cap \calA \Big) &\le 
	2 \t_j^4 
	e^{-C\min \Big\{ \frac{n\t_j^2}{\s^2}, \frac{n|\t_j|}{\s} \Big\}} \le  C'\frac{\s^4}{n^2}.
	\end{align}
	The second assertion of the lemma  follows from \eqref{eq:1term} and \eqref{eq:vspom2}.
\end{proof}

\subsubsection{Proof of Theorem~\ref{theo:pgen:sparse12}} 

Recall that, for brevity, $\hat Q_{S}^{HD}=\hat Q$.  We follow an argument close to the proof of Theorem~\ref{theo:plen:sparse}.
Similarly to~\eqref{theo:plen_sparse_1}, we first write
\begin{align}\label{th5:dec}
	\esp_{\bt,\s}\Big[\big(\hat{Q}-Q(\bt)\big)^2 \fcar_{\calA}\Big]&\le 2\esp_{\bt,\s}\Big(\fcar_{\calA} \sum_{j\in \mathcal S^\complement} a_j(\hat \bt)\fcar_{|\tilde{\t}_j|>\hat{\tau}} \Big)^2 
	\\
	&\qquad + 2\esp_{\bt,\s}\Big( \fcar_{\calA} \sum_{j\in \mathcal S} \{ a_j(\hat \bt) \fcar_{|\tilde{\t}_j|>\hat{\tau}}-\t_j^2 \} \Big)^2,
\end{align}
where $\hat{\tau}$ is defined in~\eqref{def_lem4}, $\calA$ is defined in~\eqref{def_calA}, and  $\mathcal S$ is the support of $\bt\in B_0(s)$. 
Recall that $\calA$ and $\tilde{\bt}$ are measurable with respect to $(\XX_1,\bY_1,\XX_3,\bY_3)$, $a_j(\hat \bt)$ is measurable with respect to $(\XX_1,\bY_1,\XX_2,\bY_2)$, and $\hat{\tau}$ is measurable with respect to $(\XX_1,\bY_1)$.

Following the same lines as in deriving~\eqref{for:th:5} (now, to obtain an analog of~\eqref{for:th:51} we condition  on $(\XX_1,\bY_1,\XX_3,\bY_3)$ rather than on $(\XX_1,\bY_1)$) 
we find 
\begin{align}\label{for:th:52}
&\esp_{\bt,\s} \Big[\Big( \fcar_{\calA}\sum_{j\in \mathcal S^\complement} a_j(\hat \bt)\fcar_{|\tilde{\t}_j|>\hat{\tau}}\Big)^2 \Big]\\
&\qquad \le
\frac{C}{n^2}\esp_{\bt,\s} \Big[\fcar_{\calA}\Big((\|\bu\|_2^2+\s^2)^2 \sum_{k\in \mathcal S^\complement} \fcar_{|\tilde{\t}_k|>\hat{\tau}} + \sum_{k,k'\in \mathcal S^\complement: k'\neq k} u_{k'}^2u_k^2 \fcar_{|\tilde{\t}_{k'}|>\hat{\tau}}\fcar_{|\tilde{\t}_k|>\hat{\tau}} \Big)\Big].
\end{align}
Using \eqref{for:th:52},  the definition of $\calA$ and the fact that
\begin{align}
	\sum_{k,k'\in \mathcal S^\complement: k'\neq k} u_{k'}^2u_k^2 \fcar_{|\tilde{\t}_{k'}|>\hat{\tau}}\fcar_{|\tilde{\t}_k|>\hat{\tau}}  \fcar_{\calA} \le \|\bu\|_2^4 \sum_{j \in \mathcal{S}^\complement} \fcar_{|\tilde{\t}_j|>\hat{\tau}} \fcar_{\calA} \le \s^4 \sum_{j \in \mathcal{S}^\complement} \fcar_{|\tilde{\t}_j|>\hat{\tau}} \fcar_{\calA},
\end{align}
we get the upper bound
\begin{align}
	\esp_{\bt,\s}\Big[\fcar_{\calA} \Big(\sum_{j\in \mathcal S^\complement} a_j(\hat \bt)\fcar_{|\tilde{\t}_j|>\hat{\tau}} \Big)^2\Big] &\le C\frac{\s^4}{n^2} \esp_{\bt,\s}\Big( \fcar_{\calA} \sum_{j\in \mathcal S^\complement}\fcar_{|\tilde{\t}_j|>\hat{\tau}} \Big). 
\end{align}
Noticing that for $j\in \mathcal S^\complement$ we have $\tilde{\t}_j=\nu_j+\e_j$ and applying Lemmas~\ref{eps} and~\ref{nu} we obtain
\begin{align}\label{th5:proof1}
\esp_{\bt,\s}\Big[\fcar_{\calA} \Big(\sum_{j\in \mathcal S^\complement} a_j(\hat \bt)\fcar_{|\tilde{\t}_j|>\hat{\tau}} \Big)^2\Big] & \le C\s^4 \frac{s^2}{n^2}.
\end{align}
To analyze the second term on the right hand side of~\eqref{th5:dec}, we use a decomposition analogous to~\eqref{S2_decomp} with the difference that now we insert $\fcar_{\calA}$ under the expectation. It is easy to check that the analogs of the terms $S_1$ and $S_2$ in~\eqref{S2_decomp} are bounded quite similarly to \eqref{S22} and \eqref{S21} (we skip the details here). 
This yields
\begin{align}\label{th5:proof2}
	\esp_{\bt,\s}\Big[ \fcar_{\calA} \Big(\sum_{j\in \mathcal S} \{ a_j(\hat \bt) \fcar_{|\tilde{\t}_j|>\hat{\tau}}-\t_j^2 \} \Big)^2\Big] &\le C\Big(\frac{\s^2}{n}\|\bt\|_2^2 + \s^4 \frac{s}{n^2} 
	+ \esp_{\bt,\s}\Big[\fcar_{\calA}\Big( \sum_{j \in \mathcal{S}} \t_j^2 \fcar_{\tilde{\t}_j\le \hat{\tau}} \Big)^2  \Big] \Big).
\end{align}
Now, notice that, in view of Lemmas~\ref{eps} and~\ref{nu},
	\begin{align}
 \t_j^4\prob_{\bt,\s}\big(\big\{|\tilde{\t}_j| \le \htau, |\t_j| > 2\htau\big\} \cap \calA  \big)&\le
\t_j^4\prob_{\bt,\s}\big(\big\{|\tilde{\t}_{j}-\t_j|\ge |\t_j|/2\big\} \cap \calA  \big)
\\
&\le \t_j^4\prob_{\bt,\s}\big(\big\{|\e_{j}|\ge |\t_j|/4\big\} \cap \calA  \big) + \t_j^4\prob_{\bt,\s}\big(\big\{|\nu_{j}|\ge |\t_j|/4\big\} \cap \calA  \big)
\\
&
\le C \frac{\s^4}{n^2} .
\end{align}	
Using this inequality and acting as in \eqref{S3} we get
\begin{align}
	\esp_{\bt,\s}\Big[  \fcar_{\calA} \Big( \sum_{j \in \mathcal{S}} \t_j^2 \fcar_{\tilde{\t}_j\le \hat{\tau}} \Big)^2 \Big] &\le C s^2 \esp_{\bt,\s}\big[\htau^4 \fcar_{\calA}\big] + C \Big[ \sum_{j \in \mathcal S} \t_j^2 \prob_{\bt,\s}^{1/2}\Big( \big\{|\tilde{\t}_j| \le \htau, |\t_j| > 2\htau\big\} \cap \calA \Big) \Big]^2 \\
	&\le C'\s^4\frac{s^2}{n^2}\log^2\Big(1+\frac{p}{s^2}\Big).\label{th5:proof3}
\end{align}
Combining \eqref{th5:dec}, \eqref{th5:proof1}, \eqref{th5:proof2}, and \eqref{th5:proof3} we obtain
\begin{equation}\label{th5:proof4}
	\esp_{\bt,\s}\Big[\big(\hat{Q}-Q(\bt)\big)^2 \fcar_{\calA}\Big] \le C\Big(  \s^4\frac{s^2\log^2(1+\sqrt{p}/s)}{n^2} +\s^2\frac{\|\bt\|_2^2}{n}\Big).
\end{equation}
To conclude the proof of the first inequality of the theorem (that is, the bound on the error of $\hat Q_{S}^{HD}$), it suffices to use \eqref{th5:proof4}, Markov's inequality and the fact that, due to Propositions~\ref{proposition_derumigny} and~\ref{prop_sigma_sqs}, 
\begin{equation}\label{probA}
	\prob_{\bt,\s}(\calA) \ge 1-C\exp (-(n\wedge s\log(ep/s))/C).
\end{equation} 
The  second inequality of the theorem (that is, the bound on the error of $\hat \L_{S}^{HD}$) is deduced from the first one exactly in the same way as it is done in the proof of Theorem 8 in~\cite{Collier2017}.




\section{Proofs of the lower bounds}\label{sec:proofs:lower} 

\subsection{Proof of Theorem~\ref{th:lower1}} \label{sec:proof:lower1}

It is straightforward to check that \eqref{th:lower1_3} and \eqref{th:lower1_4} follow from \eqref{th:lower1_1} - \eqref{th:lower1_2}. Thus, we only prove \eqref{th:lower1_1} - \eqref{th:lower1_2}. 

For an integer $s$ such that $1\le s \le p$ and $u>0$, we define $\Theta(s,u)=\{\bt\in B_0(s): \|\bt\|_2\ge u\}$. Note first that it suffices to prove 
\eqref{th:lower1_1} for $s\le \sqrt{p}$ and $\rho\le r$,
where for $A>0$
\begin{equation}\label{th:lower1_5}
r=A\min\left( \sqrt{\frac{s\log(1+p/s^2)}{N}},1\right).
\end{equation}
Indeed, for $\sqrt{p}< s\le p$ we have the inclusions $$\Theta(s,\s\rho(s,N,p))\supseteq \Theta(s',\s\rho(s,N,p))\supseteq \Theta(s',\s\rho(s',N,p))$$ where $s'=\lfloor \sqrt{p} \rfloor$ is the greatest integer smaller than or equal to $\sqrt{p}$. It is easy to check that $\rho(s,N,p)$ is of the same order of magnitude for $s=s'$ as for all $\sqrt{p}< s\le p$:
$$
\rho(s,N,p)\asymp \min  \Big( \frac{p^{1/4}}{\sqrt{n}}, 1  \Big), \quad \forall \ s'\le s\le p.
$$  
It follows that it is enough to prove \eqref{th:lower1_1} for $s\le s'$, and thus for $s\le \sqrt{p}$. Furthermore, if $s\le \sqrt{p}$, we have $\log(1+p/s^2)\ge \log(1+\sqrt{p}/s)$, so that it is sufficient to prove 
\eqref{th:lower1_1} for $\rho\le r$ where $r$ is defined in \eqref{th:lower1_5}.  

Thus, in the rest of this proof we assume that $s\le \sqrt{p}$ and we show that
\eqref{th:lower1_1} holds for $\bar \rho\le r$,
where $r$ is defined in \eqref{th:lower1_5}. 

Let $\tau>0$ be defined by the formula
\begin{align}\label{eq:taudef}
\tau^2= \frac{\bar \rho^2}{1+\bar \rho^2}.
\end{align}
Denote by ${\mu_{\tau}}$ the probability measure corresponding to the uniform distribution on the set $\Theta' =\Theta'(\tau)$ of all vectors in $\RR^p$ that have exactly $s$ nonzero components, all equal to $\tau/\sqrt{s}$. Note that the support of measure ${\mu_{\tau}}$ is contained in $\Theta(s,\tau)$. In the rest of this proof, we set
$$\s^2=1-\tau^2 = \frac{1}{1+\bar \rho^2}.$$
For any test $\bar\Delta$ and $\s=\sqrt{1-\tau^2}$, we have 
\begin{align}
R(\bar\Delta,s,\rho)&\ge  \prob_{\mathbf 0,1}(\bar\Delta=1) +  \sup_{\bt \in B_0(s): \|\bt\|_2\geq  \s\bar \rho} \prob_{\bt,\s}(\bar\Delta=0)\\
&\ge  \prob_{\mathbf 0,1}(\bar\Delta=1) +  \sup_{\bt \in \Theta': \|\bt\|_2=\tau  } \prob_{\bt,\sqrt{1-\tau^2}}(\bar\Delta=0)\\
&\ge  \prob_{\mathbf 0,1}(\bar\Delta=1) +  \int  \prob_{\bt, \sqrt{1-\tau^2}} (\bar\Delta=0){\mu_{\tau}}(d\bt)
= \prob_{\mathbf 0,1}(\bar\Delta=1) +  \mathbb{P}_{\mu_{\tau}} (\bar\Delta=0),\label{eq:proof_lower01}
\end{align}
where we have used the fact that $\tau=\s \bar \rho$ and the notation $\mathbb{P}_{\mu_{\tau}}(\cdot)= \int \prob_{\bt, \sqrt{1-\tau^2}}(\cdot){\mu_{\tau}}(d\bt)$. 
The theorem now follows from the last display and Lemma \ref{lem3aa} given below. 

To prove Lemma \ref{lem3aa}, we need the following two lemmas that may be of independent interest.

\begin{lemma}\label{lem3}
	Let $\tau\in (0,1)$ and $\s=\sqrt{1-\tau^2}$. Let $\bxi$ be standard normal and let ${\mathbb X}$ be a random matrix independent of $\bxi$. Then, for any $\bt, \bt'\in \RR^p$ we have
	\begin{align}
	&\int \frac{d\prob_{\bt, \s} d\prob_{\bt', \s}}{d \prob_{\mathbf 0,1}}=\l^{N/2} {\bf E}_{\mathbb X}\exp\left( \l \Big[\langle {\mathbb X}\bt,{\mathbb X}\bt'\rangle- \frac{\tau^2}{2}(\|{\mathbb X}\bt\|_2^2+\|{\mathbb X}\bt'\|_2^2) \Big] \right),
	\end{align}
	where $\l=(1-\tau^4)^{-1}$ and  ${\bf E}_{\mathbb X}$ denotes the expectation with respect to the distribution of ${\mathbb X}$.
\end{lemma}
\begin{proof} Set $a=({2}/{\s^2}-1)^{-1}>0$. We have
	\begin{align}
	&\int \frac{d\prob_{\bt, \s} d\prob_{\bt', \s}}{d \prob_{\mathbf 0,1}}= \frac{1}{(2\pi )^{N/2}\s^{2N}}{\bf E}_{\mathbb X}\int_{\RR^N} \exp\Big(-\frac{1}{2\s^2}(\|y-{\mathbb X}\bt\|_2^2 +\|y-{\mathbb X}\bt'\|_2^2)+\frac{\|y\|_2^2}{2}\Big)dy\\
	&=\frac{1}{(2\pi )^{N/2}\s^{2N}}{\bf E}_{\mathbb X}\int_{\RR^N} \exp\Big(-\frac{\|y\|_2^2}{2a}-\frac{1}{2\s^2}\Big[-2\langle y, {\mathbb X}(\bt+\bt')\rangle+\|{\mathbb X}\bt\|_2^2+\|{\mathbb X}\bt'\|_2^2\Big] \Big)dy\\
	&=\frac{1}{(2\pi )^{N/2}\s^{2N}} {\bf E}_{\mathbb X}\Big[   \exp\Big(-\frac{1}{2\s^2}\big(\|{\mathbb X}\bt\|_2^2+\|{\mathbb X}\bt'\|_2^2\big)\Big) \int_{\RR^N} \exp\Big(-\frac{\|y\|_2^2}{2a}+\frac{\langle y, {\mathbb X}(\bt+\bt')\rangle}{\s^2}\Big) dy \Big].\end{align}
	Here, 
	\begin{align}
	&\int_{\RR^N} \exp\Big(-\frac{\|y\|_2^2}{2a}+\frac{\langle y, {\mathbb X}(\bt+\bt')\rangle}{\s^2}\Big) dy
	=(2\pi a)^{N/2} \exp\Big( \frac{\| {\mathbb X}(\bt+\bt')\|_2^2}{2\s^2(2-\s^2)}\Big),
	\end{align}
	and the lemma follows.
\end{proof}

\begin{lemma}\label{lem3a}
	Under the assumptions of Lemma \ref{lem3},  let ${\mathbb X}$ be a matrix with i.i.d. standard normal entries and $\|\bt\|_2=\|\bt'\|_2=\tau$. Then, for any $\bt, \bt'\in \RR^p$ we have
	\begin{align}
	\int \frac{d\prob_{\bt, \s} d\prob_{\bt', \s}}{d \prob_{\mathbf 0,1}}& 
	=\frac{1}{(1-\langle \bt ,\bt' \rangle )^{N}}.
	\end{align}
\end{lemma}

\begin{proof}
	Lemma \ref{lem3} and the fact that  ${\mathbb X}$ is a matrix with i.i.d. standard normal entries imply
	\begin{align}\label{calcul}
	\int \frac{d\prob_{\bt, \s} d\prob_{\bt', \s}}{d \prob_{\mathbf 0,1}}
	&= \l^{N/2}\Big( {\bf E}_{\mathbb X}\exp \Big[\l \big( ({\bf X}^1 \bt ) ({\bf X}^1 \bt' )- \frac{\tau^2}{2}(({\bf X}^1 \bt)^2+ ({\bf X}^1 \bt' )^2)\big) \Big] \Big)^N
	\\
	&=  \l^{N/2} \Big( \esp \exp(\lambda Z^T {\mathbb D} Z) \Big)^N,
	\end{align}
	where  $Z\sim {\cal N}({\bf 0}, {\mathbb I}_p)$ and
	$${\mathbb D}=\frac12 \Big[(\bt\bt'^T+\bt'\bt^T)- \tau^2(\bt \bt^T+\bt'\bt'^T) \Big].$$
	The rank of the matrix ${\mathbb D}$ is at most two and, for all $\bt$ and $\bt'$ such that $\|\bt\|_2=\|\bt'\|_2=\tau$, this matrix has the eigenvalues 
	$$\l_1=\frac{\langle \bt, \bt' \rangle +\tau^2}{2}(1-\tau^2)> 0,$$
	$$\l_2=\frac{\langle \bt, \bt' \rangle -\tau^2}{2}(1+\tau^2)\le 0.$$ 
	In particular, since $\tau<1$ we have $\l_1\l \le {\tau^2}/{(1+\tau^2)} <1/2$. It follows that
	\begin{align}\label{khi2}
	\esp \exp(\lambda Z^T {\mathbb D} Z)&= \frac{1}{\sqrt{(1-2\l_1 \l)(1-2\l_2\l)}}.
	\end{align}
	using the fact that $\mathbb E_{G \sim \mathcal N(0,1)} \exp(t G^2) = (1 - 2t)^{-1/2}$ for any $t<1/2$.

	Combining~\eqref{calcul}  with~\eqref{khi2} we get after some algebra that, for all $\bt$ and $\bt'$ such that $\|\bt\|_2=\|\bt'\|_2=\tau$,
	$$
	\int \frac{d\prob_{\bt, \s} d\prob_{\bt', \s}}{d \prob_{\mathbf 0,1}}= \left(\frac{\l}{(1-2\l_1 \l)(1-2\l_2\l)}\right)^{N/2}
	=\frac{1}{(1-\langle \bt ,\bt' \rangle )^{N}}.
	$$
\end{proof}

\begin{lemma}\label{lem3aa}  Assume that ${\mathbb X}$ is a matrix with i.i.d. standard normal entries and $\bxi$ is standard normal. Let $\delta>0$. 
	Let $\tau^2=\frac{\bar \rho^2}{1+\bar \rho^2}$ as in~\eqref{eq:taudef} with $\bar \rho\le r$, where $r$ is defined in~\eqref{th:lower1_5} with $A=\sqrt{\frac12 \log\big((1-\d)^2+1\big)}$, and let ${\mu_{\tau}}$ be the prior measure defined above in this subsection. Then for any test $\bar\Delta$  we have 
	\begin{align}
	\prob_{\mathbf 0,1}(\bar\Delta=1) +  \mathbb{P}_{\mu_{\tau}} (\bar\Delta=0) \ge \delta,
	\label{eq:proof_lower02}	
	\end{align}
	where $\mathbb{P}_{\mu_{\tau}}(\cdot)= \int \prob_{\bt, \sqrt{1-\tau^2}}(\cdot){\mu_{\tau}}(d\bt)$. 
\end{lemma}
\begin{proof}
	Set $\s=\sqrt{1-\tau^2}$. It follows from Lemmas 6 and 7 in  \cite{carpentier2018minimax} that, for any test $\bar \Delta$,  
	\begin{align}\label{eq:proof_lower1}
	\prob_{\mathbf 0,1}(\bar\Delta=1) +  \mathbb{P}_{\mu_{\tau}} (\bar\Delta=0)  &\ge  1-\left(\mathbb E_{(\bt, \bt') \sim   \mu_{\tau}^2}
	\Big(\int \frac{d\prob_{\bt, \s} d\prob_{\bt', \s}}{d \prob_{\mathbf 0,1}}\Big) -1\right)^{1/2},
	\end{align}
	where $\mathbb E_{(\bt, \bt' )\sim   \mu_{\tau}^2}$ denotes the expectation with respect to the distribution of the pair $(\bt, \bt')$ such that $\bt$ and $\bt'$ are independent and each of them is distributed according to ${\mu_{\tau}}$. 
	As ${\mu_{\tau}}$ is supported on $\T(s,\tau)$ we have $\|\bt\|_2=\tau$ for all $\bt$ in the support of ${\mu_{\tau}}$.  Since $\rho \le r\le A$ we have
	$|\langle \bt ,\bt' \rangle|\le \tau^2\le A^2 \le 1/2$ for all $\bt$ and $\bt'$ in the support of ${\mu_{\tau}}$. Thus,  using Lemma \ref{lem3a} and the fact that $1-x\ge e^{-2x}$ for $0<x<1/2$ we find 
	\begin{align}\label{calcul20}
	\mathbb E_{(\bt, \bt') \sim   \mu_{\tau}^2}
	\Big(\int \frac{d\prob_{\bt, \s} d\prob_{\bt', \s}}{d \prob_{\mathbf 0,1}}\Big) 
	&\le \mathbb E_{(\bt, \bt') \sim   \mu_{\tau}^2}
	\exp \big(2N\langle \bt ,\bt' \rangle\big)\\
	&=\mathbb E_{(\bt, \bt') \sim   \mu_{\tau}^2}\exp\Big(2N\tau^2 s^{-1} \sum_{j=1}^p \fcar_{\bt_j\ne 0} \fcar_{\bt'_j\ne 0} \Big).
	\end{align}
	The last expectation involves a hypergeometric $(p,s,s)$ random variable and it is bounded in the same way as in the proof of Theorem 4 in \cite{carpentier2018minimax} (see also Lemma 1 in \cite{Collier2017}), which yields
	\begin{align}\label{calcul2}
	\mathbb E_{(\bt, \bt') \sim   \mu_{\tau}^2}
	\Big(\int \frac{d\prob_{\bt, \s} d\prob_{\bt', \s}}{d \prob_{\mathbf 0,1}}\Big) 
	&\le \exp(2A^2).
	\end{align}
	Finally, the bounds~\eqref{eq:proof_lower1}, \eqref{calcul2} and the definition of $A$ imply that
	\begin{align}
	\prob_{\mathbf 0,1}(\bar\Delta=1) +  \mathbb{P}_{\mu_{\tau}} (\bar\Delta=0) &\ge  1-\left(\exp(2A^2) -1\right)^{1/2} = \d'.
	\end{align}
	
\end{proof}

\subsection{Proof of Theorem~\ref{th:lower2} }  \label{sec:proof:lower2}

We use the following lemma. Its proof is simple and is therefore omitted. 

\begin{lemma}\label{lem:q}
	Let $\kappa>0, \e>0$, and $z\ge 1$. Then for 
	\begin{align}\label{qqq}
	q=\min(\max(\e^2 z, \e\kappa), \kappa^2)
	\end{align}
	we have
	\begin{align}
	q &=
	\left\{
	\begin{array}{lcl}
	\e\kappa & \text{if}&  \kappa^2 \ge \e^2 z^2,\\
	\e^2 z  & \text{if} &  \e^2 z < \kappa^2 \le \e^2 z^2,\\
	\kappa^2 &\text{if} & \kappa^2 < \e^2 z .
	\end{array}
	\right.
	\end{align}
\end{lemma}

To prove Theorem~\ref{th:lower2}, we first note that, as in the proof of Theorem~\ref{th:lower1}, it is enough to consider $s \le \sqrt{p}$ since for 
$s>\sqrt{p}$ the rate of estimation $q^*$ is of the same order as for $s=\lfloor\sqrt{p}\rfloor$. Next, note that for $s \le \sqrt{p}$  we have $\log(1+p/s^2)\ge\log(1+\sqrt{p}/s)$, so that  $q_*\le \min(\e^2 z + \e\kappa, \kappa^2) \le 2q$ where $q$ is given by equation~\eqref{qqq} with
$$
\e=\e(\s):=\frac{\s}{\sqrt{N}}, \quad z=z(s):= \frac{\min(s\log(1+p/s^2), N)}{\log 2},
$$
and where $\kappa$ is the bound on the $l_2$ norm of the parameter $\theta$. We have $z\ge 1$ since  $s\log(1+p/s^2)\ge \log 2$  for $1\le s\le \sqrt{p}$ - we can thus apply Lemma~\ref{lem:q}. We now prove Theorem~\ref{th:lower2} separately for the cases $\kappa^2 \ge \e^2 z^2$,  $\e^2 z < \kappa^2 \le \e^2 z^2$ and $\kappa^2 < \e^2 z$.

\paragraph{Case $\kappa^2 \ge \e^2 z^2$.} Let $c>0$. Using Lemma~\ref{lem:q} we obtain
\begin{align}
\sup_{\substack{\bt\in B_0(s): \\ \|\bt\|_2\le\kappa}} \ \sup_{\s>0} \prob_{\bt,\s}\Big(|\hat T - Q(\bt)|\ge c q_*\Big)&\ge   \sup_{\substack{\bt\in B_0(s): \\ \|\bt\|_2\le\kappa}} \ \sup_{\s>0}  \prob_{\bt,\s}\Big(|\hat T - Q(\bt)|\ge 2c \e\kappa\Big)\\
&\ge \sup_{\substack{\bt\in B_0(1): \\ \|\bt\|_2\le\kappa}} \ \prob_{\bt,1}\Big(|\hat T - Q(\bt)|\ge \frac{ 2c\kappa}{\sqrt{N}} \Big)\\
&\ge \max_{j=1,2} \prob_{\bt_j,1}\Big(|\hat T - Q(\bt_j)|\ge \frac{ 2c\kappa}{\sqrt{N}} \Big),
\end{align}
where $\bt_1=(\kappa, 0,\dots, 0)$ and $\bt_2 = (\kappa -\frac{c}{\sqrt{N}}, 0,\dots, 0)$. 

Set now $\d\in (0,1)$.  Using the above display and a standard lower bound for the maximum error of testing two hypotheses (cf., e.g., \cite[Chapter 2]{tsybakovbook}) we get that  there exists $c_\d>0$ depending only on $\d$ such that
\begin{align}
	\sup_{\substack{\bt\in B_0(s): \\ \|\bt\|_2\le\kappa}} \ \sup_{\s>0} \prob_{\bt,\s}\Big(|\hat T - Q(\bt)|\ge c_\d q_*\Big)&\ge 
	 \max_{j=1,2} \prob_{\bt_j,1}\Big(|\hat T - Q(\bt_j)|\ge \frac{ 2c_\d\kappa}{\sqrt{N}} \Big) \ge \d,
	\end{align}

\paragraph{Case $\e^2 z < \kappa^2 \le \e^2 z^2$.} Using again Lemma~\ref{lem:q} we obtain $q_*\le 2q\le  2 \e^2z$. Let $\d>0$. Recalling that $\e=\e(\s), z=z(s)$
and setting $\tau^2 = \frac{r^2}{1+r^2}, \bar \s= \sqrt{1-\tau^2} $, where $r$ is defined in equation~\eqref{th:lower1_5} with the value $A:= A(\d)$ as in Lemma \ref{lem3aa}, we have
\begin{align}
&\sup_{\substack{\bt\in B_0(s): \\ \|\bt\|_2\le\kappa}} \ \sup_{\s>0} \prob_{\bt,\s}\Big(|\hat T - Q(\bt)|\ge c q_*\Big)
\ge 
\sup_{\s>0}  \sup_{\substack{\bt\in B_0(s): \\ \|\bt\|_2\le \e(\s) \sqrt{z(s)} }} 
\  \prob_{\bt,\s}\Big(|\hat T - Q(\bt)|\ge 2c  \e^2(\s)z(s)\Big)
\\ 
&\qquad \ge  \frac12 \Big[\prob_{\mathbf 0,1}\Big(|\hat T|\ge  2c \e^2(1)z(1)\Big) + 
\sup_{\substack{\bt\in B_0(s): \\ \|\bt\|_2\le \e(\bar\s) \sqrt{z(s)} }} 
\prob_{\bt,\bar\s}\Big(|\hat T - Q(\bt)|\ge 2c \e^2(\bar\s)z(s)\Big)
\Big]
\\ 
&\qquad \ge  \frac12 \Big[\prob_{\mathbf 0,1}\Big(|\hat T|\ge  {r^2}/{4}\Big) + 
\sup_{\substack{\bt\in B_0(s): \\ \|\bt\|_2\le \e(\bar\s) \sqrt{z(s)} }} 
\prob_{\bt,\bar\s}\Big(|\hat T - Q(\bt)|\ge {r^2}/{4}\Big)
\Big]
\end{align}
if we choose $c:=c_\d>0$ such that $c\le A^2(\log 2)/8$. In the last line, we have used the relations
$$
z(1)\le z(s), \ \forall  1\le s\le \sqrt{p}, \quad   
\e^2(\bar\s) = \frac{\bar\s^2}{N}= \frac{1}{(1+r^2)N} \le  \frac{1}{N}=\e^2(1),
$$
$$
2c\e^2(1)z(s) =\frac{2cz(s)}{N} =  \frac{2c}{\log 2}  \min\Big(\frac{s\log(1+p/s^2)}{N}, 1\Big)
= \frac{2c r^2}{A^2\log 2}.
$$
Next, note that  $\{\bt\in B_0(s):  \|\bt\|_2\le \e(\bar\s) \sqrt{z(s)} \}\supset \{\bt\in B_0(s):  \|\bt\|_2\le \tau \}$. 
Indeed, 
\begin{equation}
\label{eqQ} 
\tau^2= \frac{r^2}{1+r^2}\le \frac{r^2}{(1+r^2) A^2 \log 2}  = \e^2(\bar\s)z(s),
\end{equation}
since $A<1$ (cf. definition of $A$  in Lemma \ref{lem3aa}). Thus, recalling that $\Theta'(\tau)$ is the set of all $s$-sparse vectors with $s$ components equal to $\tau/\sqrt{s}$ and all other components equal to 0 we get
\begin{align}
&\sup_{\substack{\bt\in B_0(s): \\ \|\bt\|_2\le\kappa}} \ \sup_{\s>0} \prob_{\bt,\s}\Big(|\hat T - Q(\bt)|\ge c q_*\Big)
\ge \\
&\qquad \ge  \frac12 \big[\prob_{\mathbf 0,1}\big(|\hat T|\ge  {r^2}/{4}\big) + 
\sup_{\substack{\bt\in B_0(s): \\ \|\bt\|_2\le \tau }} 
\prob_{\bt,\bar\s}\big(|\hat T - Q(\bt)|\ge {r^2}/{4}\big)
\big]
\\
&\qquad \ge  \frac12 \big[\prob_{\mathbf 0,1}\big(|\hat T|\ge  {r^2}/{4}\big) + 
\sup_{\bt\in \Theta'(\tau)} 
\prob_{\bt,\bar\s}\big(|\hat T - Q(\bt)|\ge {r^2}/{4}\big)
\big]
\\
&\qquad \ge  \frac12 \Big[\prob_{\mathbf 0,1}(\bar\Delta=1) + 
\sup_{\bt\in \Theta'(\tau)} 
\prob_{\bt,\sqrt{1-\tau^2}}(\bar\Delta=0) 
\Big]
\\
&\qquad \ge  \frac12 \Big[\prob_{\mathbf 0,1}(\bar\Delta=1) + 
\mathbb{P}_{\mu_{\tau}}(\bar\Delta=0) 
\Big]
\end{align}
where $\bar\Delta=\fcar\{|\hat T|\ge  {r^2}/{4}\}$ and the penultimate inequality is due to the fact that $Q(\bt)=\frac{r^2}{1+r^2}\ge\frac{r^2}{2} $ for all 
$\bt\in \Theta'(\tau)$. It follows from Lemma \ref{lem3aa} that the minimax risk in the last display  is greater than $\d\in (0,1)$. 

\vspace{-5mm}
\paragraph{Case $\kappa^2 < \e^2 z$.}  From Lemma~\ref{lem:q} we obtain $q_*  \le 2\kappa^2$.  Consider separately the cases $\kappa\ge \tau$ and $\kappa< \tau$. If $\kappa\ge \tau$ we have
\begin{align}
&\sup_{\substack{\bt\in B_0(s): \\ \|\bt\|_2\le\kappa}} \ \sup_{\s>0} \prob_{\bt,\s}\Big(|\hat T - Q(\bt)|\ge c q_*\Big)
\ge 
\sup_{\substack{\bt\in B_0(s): \\ \|\bt\|_2\le \kappa}}
\ \sup_{\s>0} 
\  \prob_{\bt,\s}\Big(|\hat T - Q(\bt)|\ge 2c  \e^2(\s)z(s)\Big)
\\ 
&\qquad \ge  \frac12 \Big[\prob_{\mathbf 0,1}\Big(|\hat T|\ge  2c \e^2(1)z(1)\Big) + 
\sup_{\substack{\bt\in B_0(s): \\ \|\bt\|_2\le \tau }}
\prob_{\bt,\bar\s}\Big(|\hat T - Q(\bt)|\ge 2c \e^2(\bar\s)z(s)\Big)
\Big]
\end{align}
and we use the same argument as in the case $\e^2 z < \kappa^2 \le \e^2 z^2$ to obtain the result.  On the other hand, if $\kappa< \tau$ then for any $c<1/4$ we get 
\begin{align}
&\sup_{\substack{\bt\in B_0(s): \\ \|\bt\|_2\le\kappa}} \ \sup_{\s>0} \prob_{\bt,\s}\Big(|\hat T - Q(\bt)|\ge c q_*\Big)
\ge
\sup_{\substack{\bt\in B_0(s): \\ \|\bt\|_2\le \kappa}} \ \sup_{\s>0} \prob_{\bt,\s}\Big(|\hat T - Q(\bt)|\ge 2c  \kappa^2\Big)
\\
&\qquad \ge  \frac12 \Big[\prob_{\mathbf 0,1}\Big(|\hat T|\ge  2c \kappa^2\Big) + 
\sup_{\substack{\bt\in B_0(s): \\ \|\bt\|_2\le \kappa }} 
\prob_{\bt,\sqrt{1-\kappa^2}}\Big(|\hat T - Q(\bt)|\ge 2c \kappa^2\Big)
\Big]
\\
&\qquad \ge  \frac12 \Big[\prob_{\mathbf 0,1}\Big(|\hat T|\ge  2c \kappa^2\Big) + 
\sup_{\bt\in\Theta'(\kappa)} 
\prob_{\bt,\sqrt{1-\kappa^2}}\Big(|\hat T - Q(\bt)|\ge 2c \kappa^2\Big)
\Big]
\\
&\qquad \ge  \frac12 \Big[\prob_{\mathbf 0,1}(\tilde\Delta=1) +
\sup_{\bt\in \Theta'(\kappa)} 
\prob_{\bt,\sqrt{1-\kappa^2}}(\tilde\Delta=0) 
\Big]
\\
&\qquad \ge  \frac12 \Big[\prob_{\mathbf 0,1}(\tilde\Delta=1) + 
\mathbb{P}_{\mu_{\kappa}}(\tilde\Delta=0) 
\Big],
\end{align}
where $\tilde\Delta=\fcar\{|\hat T|\ge  {\kappa^2}/{2}\}$  and the penultimate inequality is due to the fact that $Q(\bt)=\kappa^2 $ for all 
$\bt\in \Theta'(\kappa)$. The result now follows by applying Lemma \ref{lem3aa} where we replace $\tau$ by $\kappa$ and take $\bar \rho\in (0,r)$ such that $\kappa^2=\frac{\bar \rho^2}{1+\bar \rho^2}$.

\section*{Acknowledgments}

The work of O. Collier was supported by the French National Research Agency (ANR) under the grant Labex MME-DII (ANR-11-LBX-0023-01).  The work of A.B.Tsybakov was supported by GENES and by ANR under the grant Labex Ecodec (ANR-11-LABEX-0047). 
The work of A. Carpentier is partially supported by the Deutsche Forschungsgemeinschaft (DFG) Emmy Noether grant MuSyAD (CA 1488/1-1), by the DFG - 314838170, GRK 2297 MathCoRe, by the DFG GRK 2433 DAEDALUS (384950143/GRK2433), by the DFG CRC 1294 'Data Assimilation', Project A03, and by the UFA-DFH through the French-German Doktorandenkolleg CDFA 01-18 and by the UFA-DFH through the French-German Doktorandenkolleg CDFA 01-18 and by the SFI Sachsen-Anhalt for the project RE-BCI.

\bibliographystyle{plain}
\small{
\bibliography{biblio}

\begin{thebibliography}{10}

\bibitem{2015_adamszak}
Radoslaw Adamczak and Pawel Wolff.
\newblock Concentration inequalities for non-{L}ipschitz functions with bounded
  derivatives of higher order.
\newblock {\em Probab. Theory Related Fields}, 162(3-4):531--586, 2015.

\bibitem{2011_AS_Arias-Castro}
Ery Arias-Castro, Emmanuel Candes, and Yaniv Plan.
\newblock Global testing under sparse alternatives: Anova, multiple comparisons
  and the higher criticism.
\newblock {\em Ann. Statist.}, 39(5):2533--2556, 2011.

\bibitem{baraud02}
Yannick Baraud.
\newblock Non-asymptotic minimax rates of testing in signal detection.
\newblock {\em Bernoulli}, 8(5):577--606, 2002.

\bibitem{BellecLecueTsybakov2017}
Pierre Bellec, Guillaume Lécué, and Alexandre~B. Tsybakov.
\newblock Slope meets {L}asso: Improved oracle bounds and optimality.
\newblock {\em Ann. Statist.}, 46(6B):3603--3642, 2018.

\bibitem{MR2253108}
Tony~T. Cai and Mark~G. Low.
\newblock Nonquadratic estimators of a quadratic functional.
\newblock {\em Ann. Statist.}, 33(6):2930--2956, 2005.

\bibitem{carpentier2018minimax}
Alexandra Carpentier, Olivier Collier, La\"etitia Comminges, Alexandre~B.
  Tsybakov, and Yuhao Wang.
\newblock Minimax rate of testing in sparse linear regression.
\newblock {\em Automation and Remote Control}, 80(10):1817--1834, 2019.

\bibitem{carpentier-verzelen2019a}
Alexandra Carpentier and Nicolas Verzelen.
\newblock Optimal sparsity testing in linear regression model.
\newblock {\em arXiv preprint arXiv:1901.08802 (version 1)}, 2019.

\bibitem{Collier2017}
Olivier Collier, La{\"e}titia Comminges, and Alexandre~B. Tsybakov.
\newblock Minimax estimation of linear and quadratic functionals on sparsity
  classes.
\newblock {\em Ann. Statist.}, 45(3):923--958, 2017.

\bibitem{collier2016optimal}
Olivier Collier, La{\"e}titia Comminges, Alexandre~B. Tsybakov, and Nicolas
  Verzelen.
\newblock Optimal adaptive estimation of linear functionals under sparsity.
\newblock {\em Ann. Statist.}, 46(6A):3130--3150, 2016.

\bibitem{CCNT2018}
La{\"e}titia Comminges, Olivier Collier, Mohamed Ndaoud, and Alexandre~B.
  Tsybakov.
\newblock Adaptive robust estimation in sparse vector model.
\newblock {\em arXiv preprint arXiv:1802.04230}, 2018.

\bibitem{comminges2013minimax}
La{\"e}titia Comminges and Arnak~S. Dalalyan.
\newblock Minimax testing of a composite null hypothesis defined via a
  quadratic functional in the model of regression.
\newblock {\em Electronic J. of Statist.}, 7:146--190, 2013.

\bibitem{Derumigny2017}
Alexis Derumigny.
\newblock Improved bounds for square-root lasso and square-root slope.
\newblock {\em Electronic J. of Statist.}, 12:741--766, 2017.

\bibitem{jin2004}
David~L. Donoho and Jiashun Jin.
\newblock Higher criticism for detecting sparse heterogeneous mixtures.
\newblock {\em Ann. Statist.}, 32(3):962--994, 2004.

\bibitem{donoho_nussbaum}
David~L. Donoho and Michael Nussbaum.
\newblock Minimax quadratic estimation of a quadratic functional.
\newblock {\em J. Complexity}, 6(3):290--323, 1990.

\bibitem{CaiGuo2016genetic_relatedness}
Zijian Guo, Weinjie Wang, Cai~Tony T., and Henghze Li.
\newblock Optimal estimation of genetic relatedness in high-dimensional linear
  models.
\newblock {\em J. Amer. Stat. Assoc.}, 114(525):358--369, 2019.

\bibitem{ingster1997}
Yuri~I. Ingster.
\newblock Some problems of hypothesis testing leading to infinitely divisible
  distributions.
\newblock {\em Math. Meth. Stat.}, 6:47--69, 1997.

\bibitem{ingster_suslina}
Yuri~I. Ingster and Irina~A. Suslina.
\newblock {\em Nonparametric goodness-of-fit testing under {G}aussian models},
  volume 169 of {\em Lect. Notes Stat.}
\newblock Springer-Verlag, New York, 2003.

\bibitem{2010_EJS_Ingster}
Yuri~I. Ingster, Alexandre~B. Tsybakov, and Nicolas Verzelen.
\newblock Detection boundary in sparse regression.
\newblock {\em Electronic J. of Statist.}, 4:1476--1526, 2010.

\bibitem{LecueMendelson2017}
Guillaume Lecué and Mendelson Shahar.
\newblock Sparse recovery under weak moment assumptions.
\newblock {\em J. Eur. Math. Soc.}, 19(3):881--904, 2017.

\bibitem{mukherjee2020}
Rajarshi Mukherjee and Subhabrata Sen.
\newblock On minimax exponents of sparse testing.
\newblock {\em arXiv preprint arXiv:2003.00570}, 2020.

\bibitem{RudelsonVershynin2015}
Mark Rudelson and Roman Vershynin.
\newblock Small ball probabilities for linear images of high-dimensional
  distributions.
\newblock {\em Inter. Math. Res. Not.}, (19):9594--9617, 2015.

\bibitem{2012_Sun}
Tingni Sun and Cun-Hui Zhang.
\newblock Scaled sparse linear regression.
\newblock {\em Biometrika}, 99(4):879--898, 2012.

\bibitem{tsybakovbook}
Alexandre~B. Tsybakov.
\newblock {\em Introduction to Nonparametric Estimation}.
\newblock Springer, New York, 2009.

\bibitem{vershyninbook}
Roman Vershynin.
\newblock {\em High-Dimensional Probability}.
\newblock Springer Series in Statistics. Springer, New York, 2018.

\bibitem{verzelen_minimax}
Nicolas Verzelen.
\newblock Minimax risks for sparse regressions: ultra-high dimensional
  phenomenons.
\newblock {\em Electronic J. of Statistics}, 6:38--90, 2012.

\bibitem{verzelen-gassiat}
Nicolas Verzelen and Elisabeth Gassiat.
\newblock Adaptive estimation of high-dimensional signal-to-noise ratios.
\newblock {\em Bernoulli}, 24(4B):3683--3710, 2018.

\end{thebibliography}
}

\end{document}